\documentclass[11pt,letterpaper]{article}
\usepackage[margin=1in]{geometry}

\usepackage{amsmath,amssymb,amsthm,mathrsfs,bbm,comment,units,enumitem,xcolor}
\usepackage[colorlinks, linktocpage, allcolors=black,breaklinks]{hyperref}
\usepackage{pgf,tikz}
\usepackage{tikz-3dplot}
\usetikzlibrary{positioning}
\usetikzlibrary{decorations.pathreplacing}
\usetikzlibrary{arrows,shapes}
\usetikzlibrary{arrows.meta,bending,automata}

\renewcommand{\leq}{\leqslant}
\renewcommand{\geq}{\geqslant}
\newcommand{\concat}{\overset{\frown}{}}

\newcommand{\One}{\text{I}}
\newcommand{\Two}{\text{II}}
\newcommand{\cof}{\text{cof}}

\theoremstyle{definition}
\newtheorem{theorem}{Theorem}
\newtheorem{definition}{Definition}
\newtheorem{lemma}[theorem]{Lemma}
\newtheorem{corollary}[theorem]{Corollary}
\newtheorem{proposition}[theorem]{Proposition}

\newtheorem{remark}{Remark}
\newtheorem{note}{Note}
\newtheorem*{notn}{Notation}

\title{Selection Games on Continuous Functions}
\author{Christopher Caruvana and Jared Holshouser}
\date{\today}

\begin{document}

\maketitle

\begin{abstract}
    In this paper we study the selection principle of closed discrete selection, first researched by Tkachuk in \cite{Tkachuk2018} and strengthened by Clontz, Holshouser in \cite{ClontzHolshouser}, in set-open topologies on the space of continuous real-valued functions.
    Adapting the techniques involving point-picking games on \(X\) and \(C_p(X)\), the current authors showed similar equivalences in \cite{CaruvanaHolshouser} involving the compact subsets of \(X\) and \(C_k(X)\).
    By pursuing a bitopological setting, we have touched upon a unifying framework which involves three basic techniques: general game duality via reflections (Clontz), general game equivalence via topological connections, and strengthening of strategies (Pawlikowski and Tkachuk).
    Moreover, we develop a framework which identifies topological notions to match with generalized versions of the point-open game.
\end{abstract}

\section{Introduction}

The closed discrete selection principle was first studied by Tkachuk in 2017.
This property occurs naturally in the course of studying functional analysis.
Tkachuk connected this selection principle on \(C_p(X)\) with topological properties of \(X\).
He then went on to consider the corresponding selection game, creating a partial characterization of winning strategies in that game and finding connections between it, the point-open game on \(X\), and Gruenhage's \(W\)-game on \(C_p(X)\) \cite{TkachukGame}.
In 2019, Clontz and Holshouser \cite{ClontzHolshouser} finished this characterization, showing that the discrete selection game on \(C_p(X)\) is equivalent to a modification of the point-open game on \(X\).
Clontz and Holshouser show this not only for full information strategies but also for limited information strategies.

The current authors continued this work, researching the closed discrete game on \(C_k(X)\), the real-valued continuous functions with the compact open topology \cite{CaruvanaHolshouser}.
They show that similar connections exist in this setting, with the point-open game on \(X\) replaced by the compact-open game.
They also isolated general techniques which have use beyond the study of closed discrete selections.

In this paper, we study the problem of closed discrete selection in the general setting of set-open topologies on the space of continuous functions.
We use closed discrete selection as a tool not only for comparing \(X\) to its space of continuous functions, but also for comparing different set-open topologies to each other.
To establish these connections, we prove general statements in three categories:
\begin{enumerate}
    \item strengthening the strategies in games,
    \item criteria for games to be dual,
    \item characterizations of strong strategies in abstract point-open games,
\end{enumerate}
and use work of Clontz \cite{ClontzDuality} to show that some general classes of games are equivalent.

In this version, we have
\begin{itemize}
    \item
    identified that Lemmas \ref{lemma:PawlikowskiA} and \ref{lemma:Pawlikowski}, as stated, are not known to be true yet and reference \href{https://arxiv.org/abs/2102.00296}{arXiv:2102.00296} for a revised proof for \(k\)-covers.
    \item
    corrected a slight error in the statement of Theorem \ref{MD2} which relied on Lemma \ref{lemma:Pawlikowski}.
\end{itemize}

\section{Definitions and Preliminaries}

\begin{definition}
    Let \(X\) be a space and \(\mathcal A \subseteq \wp(X)\).
    We say that \(\mathcal A\) is an \textbf{ideal-base} if, for \(A_1, A_2 \in \mathcal A\), there exists \(A_3 \in \mathcal A\) so that \(A_1 \cup A_2 \subseteq A_3\).
\end{definition}

\begin{definition}
    For a topological space \(X\) and a collection \(\mathcal A \subseteq \wp(X)\), we let \(\bar{\mathcal A} = \{ \text{cl}_X(A) : A \in \mathcal A \}\).
\end{definition}

\begin{definition}
    Fix a topological space \(X\) and a collection \(\mathcal A \subseteq \wp(X)\). Then
    \begin{itemize}
        \item we let \(C_p(X)\) denote the set of all continuous functions \(X \to \mathbb R\) endowed with the topology of point-wise convergence; we also let \(\mathbf 0\) be the function which identically zero.
        \item we let \(C_k(X)\) denote the set of all continuous functions \(X \to \mathbb R\) endowed with the topology of uniform convergence on compact subsets of \(X\); we will write
	    \[
	    [f;K,\varepsilon] = \left\{ g \in C_k(X) : \sup\{ |f(x)-g(x)| : x \in K \} < \varepsilon \right\}
	    \]
	    for \(f \in C_k(X)\), \(K \subseteq X\) compact, and \(\varepsilon > 0\),
	    \item in general, we let \(C_{\mathcal A}(X)\) denote the set of all continuous functions \(X \to \mathbb R\) endowed with the \(\mathcal A\)-open topology; we will write
	    \[
	    [f;A,\varepsilon] = \left\{ g \in C_{\mathcal A}(X) : \sup\{ |f(x)-g(x)| : x \in A \} < \varepsilon \right\}
	    \]
	    for \(f \in C_{\mathcal A}(X)\), \(A \in \mathcal A\), and \(\varepsilon > 0\),
    \end{itemize}
    Notice that, for the sets of the form \([f;A,\varepsilon]\) to be a base for the topology \(C_{\mathcal A}(X)\), then \(\mathcal A\) must be an ideal-base.
\end{definition}
\begin{definition}
	For a topological space \(X\), we let \(K(X)\) denote the family of all non-empty compact subsets of \(X\).
\end{definition}
\begin{definition}
    Let \(X\) be a topological space.
    We say that \(A \subseteq X\) is \textbf{\(\mathbb R\)-bounded} if, for every continuous \(f : X \to \mathbb R\), \(f[A]\) is bounded.
\end{definition}

In this paper, we will be concerned with selection principles and related games.
For classical results, basic tools, and notation, the authors recommend \cite{SakaiScheppers} and \cite{KocinacSelectedResults}.

\begin{definition}
    Consider collections \(\mathcal A\) and \(\mathcal B\) and an ordinal \(\alpha\).
    The corresponding selection principles are defined as follows:
    \begin{itemize}
        \item \(S_{\text{fin}}^\alpha(\mathcal A, \mathcal B)\) is the assertion that, given any \(\{A_\xi : \xi \in \alpha\} \subseteq \mathcal A\), there exists \(\{ \mathcal F_\xi : \xi \in \alpha \}\) so that, for each \(\xi \in \alpha\), \(\mathcal F_\xi\) is a finite subset of \(A_\xi\) (denoted as \(\mathcal F_\xi \in [A_\xi]^{<\omega}\) hereinafter) and \(\bigcup\{ \mathcal F_\xi : \xi \in \alpha \} \in \mathcal B\), and
        \item \(S_1^\alpha(\mathcal A, \mathcal B)\) is the assertion that, given any \(\{A_\xi : \xi \in \alpha\} \subseteq \mathcal A\), there exists \(\{ x_\xi : \xi \in \alpha \}\) so that, for each \(\xi \in \alpha\), \( x_\xi \in A_\xi\) and \(\{ x_\xi : \xi \in \alpha \} \in \mathcal B\).
    \end{itemize}
    We suppress the superscript when \(\alpha = \omega\); i.e., \(S_1(\mathcal A, \mathcal B) = S^\omega_1(\mathcal A, \mathcal B)\).
\end{definition}

\begin{definition}
	Let \(X\) be a topological space and \(\mathscr U\) be an open cover of \(X\) with \(X \notin \mathscr U\). Recall that
	\begin{itemize}
	    \item \(\mathscr U\) is said to be a \textbf{\(\Lambda\)-cover} if, for every \(x \in X\), \(\{ U \in \mathscr U : x \in U \}\) is infinite,
	    \item \(\mathscr U\) is an \textbf{\(\omega\)-cover} of \(X\) provided that given any finite subset \(F\) of \(X\), there exists some \(U \in \mathscr U\) so that \(F \subseteq U\),
	    \item \(\mathscr U\) is said to be a \textbf{\(\gamma\)-cover} if \(\mathscr U\) is an infinite \(\omega\)-cover and for every finite subset \(F \subseteq X\), \(\{ U \in \mathscr U : F \not\subseteq U \}\) is finite,
	    \item \(\mathscr U\) is a \textbf{\(k\)-cover} of \(X\) provided that given any compact subset \(K\) of \(X\), there exists some \(U \in \mathscr U\) so that \(K \subseteq U\), and
	    \item \(\mathscr U\) is said to be a \textbf{\(\gamma_k\)-cover} if \(\mathscr U\) is an infinite \(k\)-cover and for every compact \(K \subseteq X\), \(\{ U \in \mathscr U : K \not\subseteq U \}\) is finite.
	\end{itemize}
	Note that if \(\mathscr U = \{U_n : n \in \omega\}\), then \(\mathscr U\) is a \(\gamma_k\)-cover if and only if every cofinal sequence of the \(U_n\) form an \(k\)-cover.

	For a family of sets \(\mathcal A\), let
	\begin{itemize}
	    \item \(\mathcal O(X, \mathcal A)\) to be all open covers \(\mathscr U\) so that \(X\not\in\mathscr U\) and for every \(A \in \mathcal{A}\), there is an open set \(U \in \mathscr U\) which contains \(A\),
        \item \(\Lambda(X, \mathcal A)\) be all open covers \(\mathscr U\) so that \(X\not\in\mathscr U\), and for all \(A \in \mathcal A\), there are infinitely many \(U \in \mathscr U\) so that \(A \subseteq U\), and
        \item \(\Gamma(X, \mathcal A)\) to be all infinite open covers \(\mathscr U\) so that \(X\not\in\mathscr U\) and for every \(A \in \mathcal{A}\), \(\{ U \in \mathscr U : A \not\subseteq U\}\) is finite.
		\end{itemize}
\end{definition}

\begin{remark}
    Note that
    \begin{itemize}
        \item
		\(\mathcal O(X,[X]^{<\omega}) = \Omega_X\) denotes the collection of all \(\omega\)-covers of \(X\).
		\item
		\(\mathcal O(X, K(X)) = \mathcal K_X\) denotes the collection of all \(k\)-covers of \(X\).
		\item
		\(\Gamma(X, K(X)) = \Gamma_k(X)\) denotes the collection of all \(\gamma_k\)-covers of \(X\).
    \end{itemize}
\end{remark}

\begin{notn}
	We let
	\begin{itemize}
	    \item
	    For any collection \(\mathcal A\), \(\neg \mathcal A\) is the complement of \(\mathcal A\).
	    \item
	    \(\mathscr T_X\) denote the set of all non-empty subsets of \(X\).
	    \item
	    \(\Omega_{X,x}\) denote the set of all \(A \subseteq X\) with \(x \in \text{cl}_X(A)\).
	    We also call \(A \in \Omega_{X,x}\) a \textbf{blade} of \(x\).
	    \item
	    \(\Gamma_{X,x}\) denote the set of all sequences \(\{x_n : n \in \omega\} \subseteq X\) with \(x_n \to x\).
	    \item
	    \(\mathcal D_X\) denote the collection of all dense subsets of \(X\).
	    \item
	    \(\text{CD}_X\) denote the collection of all closed and discrete subsets of \(X\).
		\item
		\(\mathcal O_X\) denote the collection of all open covers of \(X\).
		\item
		\(\Lambda_X\) denote the collection of all \(\lambda\)-covers of \(X\).
		\item
		\(\Gamma_X\) denote the collection of all \(\gamma\)-covers of \(X\).
	\end{itemize}
\end{notn}

We can create variations of selection principles and their negations by looking at selection games.

\begin{definition}
	Given a set \(\mathcal A\) and another set \(\mathcal B\), we define the \textbf{finite selection game} \(G^\alpha_{\text{fin}}(\mathcal A, \mathcal B)\) for \(\mathcal A\) and \(\mathcal B\) as follows:
	\[
		\begin{array}{c|cccccc}
			\text{I} & A_0 & A_1 & A_2 & \cdots & A_\xi & \cdots\\
			\hline
			\text{II} & \mathcal F_0 & \mathcal F_1 & \mathcal F_2 & \cdots & \mathcal F_\xi & \cdots
		\end{array}
	\]
	where \(A_\xi \in \mathcal A\) and \(\mathcal F_\xi \in [A_\xi]^{<\omega}\) for all \(\xi < \alpha\).
	We declare Two the winner if \(\bigcup\{ \mathcal F_\xi : \xi < \alpha \} \in \mathcal B\).
	Otherwise, One wins.
	We let \(G_{\text{fin}}(\mathcal A, \mathcal B)\) denote \(G^\omega_{\text{fin}}(\mathcal A, \mathcal B)\).
\end{definition}

\begin{definition}
	Similarly, we define the \textbf{single selection game} \(G^\alpha_1(\mathcal A, \mathcal B)\) as follows:
	\[
		\begin{array}{c|cccccc}
			\text{I} & A_0 & A_1 & A_2 & \cdots & A_\xi & \cdots\\
			\hline
			\text{II} & x_0 & x_1 & x_2 & \cdots & x_\xi & \cdots
		\end{array}
	\]
	where each \(A_\xi \in \mathcal A\) and \(x_\xi \in A_\xi\).
	We declare Two the winner if \(\{ x_\xi : \xi \in \alpha \} \in \mathcal B\).
	Otherwise, One wins.
	We let \(G_{1}(\mathcal A, \mathcal B)\) denote \(G^\omega_{1}(\mathcal A, \mathcal B)\).
\end{definition}

\begin{definition}
    We define strategies of various strength below.
    \begin{itemize}
    \item A \textbf{strategy for player One} in \(G^\alpha_1(\mathcal A, \mathcal B)\) is a function \(\sigma:(\bigcup \mathcal A)^{<\alpha} \to \mathcal A\).
    A strategy \(\sigma\) for One is called \textbf{winning} if whenever \(x_\xi \in \sigma\langle x_\zeta : \zeta < \xi \rangle\) for all \(\xi < \alpha\), \(\{x_\xi:\xi \in \alpha\} \not\in \mathcal B\).
    If player One has a winning strategy, we write \(\One \uparrow G^\alpha_1(\mathcal A, \mathcal B)\).
    \item A strategy for player Two in \(G^\alpha_1(\mathcal A, \mathcal B)\) is a function \(\tau:\mathcal A^{<\alpha} \to \bigcup \mathcal A\).
    A strategy \(\tau\) for Two is \textbf{winning} if whenever \(A_\xi \in \mathcal A\) for all \(\xi < \alpha\), \(\{\tau(A_0,\cdots,A_\xi) : \xi < \alpha\} \in \mathcal B\).
    If player Two has a winning strategy, we write \(\Two \uparrow G^\alpha_1(\mathcal A, \mathcal B)\).
    \item A \textbf{predetermined strategy} for One is a strategy which only considers the current turn number.
    We call this kind of strategy predetermined because One is not reacting to Two's moves, they are just running through a pre-planned script. Formally it is a function \(\sigma:\alpha \to \mathcal A\).
    If One has a winning predetermined strategy, we write \(\One \underset{\text{pre}}{\uparrow} G^\alpha_1(\mathcal A, \mathcal B)\).
    \item A \textbf{Markov strategy} for Two is a strategy which only considers the most recent move of player One and the current turn number.
    Formally it is a function \(\tau:\mathcal A \times \alpha \to \bigcup \mathcal A\).
    If Two has a winning Markov strategy, we write \(\Two \underset{\text{mark}}{\uparrow} G^\alpha_1(\mathcal A, \mathcal B)\).
    \end{itemize}
\end{definition}

\begin{definition}
    Two games \(\mathcal G_1\) and \(\mathcal G_2\) are said to be \textbf{strategically dual} provided that the following two hold:
    \begin{itemize}
        \item \(\text{I} \uparrow \mathcal G_1 \text{ iff } \text{II} \uparrow \mathcal G_2\)
        \item \(\text{I} \uparrow \mathcal G_2 \text{ iff } \text{II} \uparrow \mathcal G_1\)
    \end{itemize}
    Two games \(\mathcal G_1\) and \(\mathcal G_2\) are said to be \textbf{Markov dual} provided that the following two hold:
    \begin{itemize}
        \item \(\text{I} \underset{\text{pre}}{\uparrow} \mathcal G_1 \text{ iff } \text{II} \underset{\text{mark}}{\uparrow} \mathcal G_2\)
        \item \(\text{I} \underset{\text{pre}}{\uparrow} \mathcal G_2 \text{ iff } \text{II} \underset{\text{mark}}{\uparrow} \mathcal G_1\)
    \end{itemize}
    Two games \(\mathcal G_1\) and \(\mathcal G_2\) are said to be \textbf{dual} provided that they are both strategically dual and Markov dual.
\end{definition}

\begin{remark}
    In general, \(S_1^\alpha(\mathcal A, \mathcal B)\) holds if and only if \(\text{I} \underset{\text{pre}}{\not\uparrow} G_1^\alpha(\mathcal{A},\mathcal{B})\).
    See \cite[Prop. 13]{ClontzHolshouser}.
\end{remark}

\begin{remark}
	The game \(G_{\text{fin}}(\mathcal O_X,\mathcal O_X)\) is the well-known Menger game and the game \(G_1(\mathcal O_X, \mathcal O_X)\) is the well-known Rothberger game.
\end{remark}

\begin{notn}
	For \(A \subseteq X\), let \(\mathscr N(A)\) be all open sets \(U\) so that \(A \subseteq U\).
	Set \(\mathscr N[X] = \{\mathscr N_x :x \in X\}\), and in general if \(\mathcal A\) is a collection of subsets of \(X\), then \(\mathscr N[\mathcal A] = \{\mathscr N(A) :A \in \mathcal{A}\}\).
	In the case when \(X\) and \(X^\prime\) represent two topologies on the same underlying set, we will use the notation \(\mathscr N_X(A)\) to denote the collection of open sets relative to the topology according to \(X\) that contain \(A\).
\end{notn}

\begin{remark}
	The game \(G_1(\mathscr N[X],\neg \mathcal O_X)\) is the well-known point-open game first appearing in \cite{Galvin1978}: player One is trying to build an open cover and player Two is trying to avoid building an open cover. The game \(G_1(\mathscr N[K(X)], \neg \mathcal O_X)\) is the compact-open game.

	Generally, when \(\mathscr N[\mathcal A]\) is being used in a game, we will use the identification of \(A\) with \(\mathscr N(A)\) to simplify notation.
    Particularly, One picks \(A \in \mathcal A\) and Two's response will be an open set \(U\) so that \(A \subseteq U\).
\end{remark}

\begin{definition}
    A topological space \(X\) is called \textbf{discretely selective} if, for any sequence \(\{ U_n : n \in \omega \}\) of non-empty open sets, there exists a closed discrete set \(\{x_n : n \in \omega\} \subseteq X\) so that \(x_n \in U_n\) for each \(n \in \omega\); i.e. \(S_1(\mathscr T_X, {\text{CD}}_X)\) holds.
    This notion was first isolated by Tkachuk in \cite{Tkachuk2018}.
\end{definition}
\begin{definition}
	\label{definition:ClosedDiscrete}
	For a topological space \(X\), the \textbf{closed discrete selection game} on \(X\), is \(G_1(\mathscr T_X, {\text{CD}}_X)\).
	Tkachuk studies this game in \cite{TkachukGame}.
\end{definition}
Note that \(X\) is discretely selective if and only if \(\text{I} \underset{\text{pre}}{\not\uparrow} G_1(\mathscr T_X, {\text{CD}}_X)\).

\begin{remark}
	\label{definition:GruenhageGame}
	For a topological space \(X\) and \(x\in X\), \textbf{Gruenhage's \(W\)-game} for \(X\) at \(x\) is \(G_1(\mathscr N(x), \neg \Gamma_{X,x})\) and \textbf{Gruenhage's clustering game} for \(X\) at \(x\) is \(G_1(\mathscr N(x), \neg \Omega_{X,x})\).
\end{remark}

\begin{definition}
    Suppose \((P, \leq)\) is a partially ordered set and \(\mathcal A, \mathcal B \subseteq P\). Then \textbf{\(\mathcal A\) has cofinality \(\kappa\) relative to \(\mathcal B\)}, denoted
    \[
    \mbox{cof}(\mathcal A; \mathcal B, \leq) = \kappa,
    \]
    if \(\kappa\) is the minimum cardinal so that there is a collection \(\{A_\alpha : \alpha < \kappa\} \subseteq \mathcal A\) with the property that whenever \(B \in \mathcal B\), there is an \(\alpha\) so that \(B \leq A_\alpha\). If there is no such cardinal don't define the cofinality.
\end{definition}

\begin{definition}
    Suppose \((P, \leq)\) and \((Q, \leq^*)\) are partial orders and \(\mathcal A, \mathcal B \subseteq P\), \(\mathcal C, \mathcal D \subseteq Q\). Then
    \[
    (\mathcal A; \mathcal B, \leq) \geq_T (\mathcal C; \mathcal D, \leq^*)
    \]
    if there is a map \(\varphi:\mathcal A \to \mathcal C\) so that whenever \(\mathcal F \subseteq \mathcal A\) is cofinal relative to \(\mathcal B\), then \(\varphi[\mathcal F]\) is cofinal relative to \(\mathcal D\).
    This definition is inspired by Paul Gartside and Ana Mamatelashvili's work on the Tukey order \cite{Gartside}.
\end{definition}

Suppose \((P , \leq)\) is a partially ordered set.
We define \(\leq\) on \(P \times \omega\) by
\[
    (p , n) \leq (q, m) \Longleftrightarrow (p \leq q \text{ and } n \leq m).
\]

\begin{lemma}
    For any partially ordered set \((P, \leq)\) and any \(Q \subseteq P\), \((Q \times \omega,P \times \omega) \geq_T (Q , P)\).
\end{lemma}
\begin{proof}
    Let \(\phi : P \times \omega \to P\) be defined by \(\phi(p,n) = p\).
    Suppose \(A \subseteq P \times \omega\) is cofinal for \(Q \times \omega\) and let \(q \in Q\) be arbitrary.
    By the cofinality of \(A\), we can find \((r,m) \in A\) so that \((q,0) \leq (r,m)\).
    It follows that \(q \leq r = \phi(r,m)\) which demonstrates that \(\phi[A]\) is cofinal for \(Q\).
\end{proof}

\begin{lemma}
    Suppose \((P, \leq)\) and \((Q, \leq^*)\) are partial orders, \(\mathcal A, \mathcal B \subseteq P\), and \(\mathcal C, \mathcal D \subseteq Q\). Suppose further that \((\mathcal A; \mathcal B, \leq) =_T (\mathcal C; \mathcal D, \leq^*)\) and \(\cof(\mathcal A; \mathcal B, \leq) = \kappa\). Then \(\cof(\mathcal C; \mathcal D, \leq^*) = \kappa\).
\end{lemma}
\begin{proof}
    Let \(\varphi:\mathcal A \to \mathcal C\) be so that whenever \(\mathcal F \subseteq \mathcal A\) is cofinal for \(\mathcal B\), then \(\varphi[\mathcal F]\) is cofinal for \(\mathcal D\).
    Also let \(\mathcal F = \{A_\alpha : \alpha < \kappa\} \subseteq \mathcal A\) be cofinal for \(\mathcal B\).
    Then \(\varphi[\mathcal F]\) is a subset of \(\mathcal C\) and is cofinal for \(\mathcal D\).
    Thus \(\cof(\mathcal C; \mathcal D, \leq^*) \leq \kappa\).

    Suppose towards a contradiction that \(\cof(\mathcal C; \mathcal D, \leq^*) = \lambda < \kappa\).
    Then we can find a collection \(\mathcal G = \{C_\alpha : \alpha < \lambda\} \subseteq \mathcal C\) which is cofinal for \(\mathcal D\).
    Now let \(\psi:\mathcal C \to \mathcal A\) witness that \((\mathcal C; \mathcal D, \leq^*) \geq_T (\mathcal A; \mathcal B, \leq)\).
    Then \(\psi[\mathcal G] \subseteq \mathcal A\) and is cofinal for \(\mathcal B\).
    But this would imply that \(\cof(\mathcal A; \mathcal B, \leq) < \kappa\), a contradiction.
\end{proof}

\begin{lemma} \label{lemma:CofinalityBetweenGroundAndFunctions}
    Suppose \(X\) is a Tychonoff space.
    Assume \(\mathcal A, \mathcal B \subseteq \wp(X)\).
    Then
    \[
    (\mathscr N_{C_{\mathcal A}(X)}(\mathbf 0); \mathscr N_{C_{\mathcal{B}}(X)}(\mathbf 0), \supseteq) \leq_T (\mathcal A \times \omega; \mathcal B \times \omega, \subseteq)
    \]
    and
    \[
    (\mathscr N_{C_{\bar{\mathcal A}}(X)}(\mathbf 0); \mathscr N_{C_{\mathcal{B}}(X)}(\mathbf 0), \supseteq) =_T (\bar{\mathcal A} \times \omega; \mathcal B \times \omega, \subseteq).
    \]
\end{lemma}
\begin{proof}
To address \((\mathscr N_{C_{\mathcal A}(X)}(\mathbf 0); \mathscr N_{C_{\mathcal{B}}(X)}(\mathbf 0), \supseteq) \leq_T (\mathcal A \times \omega; \mathcal B \times \omega, \subseteq)\), define \(\psi : \mathcal A \times \omega \to \mathscr N_{C_{\mathcal A}}(\mathbf 0)\) by
\[
    \psi(A,n) = [\mathbf 0; A, 2^{-n}].
\]
Suppose \(\mathcal F \subseteq \mathcal A \times \omega\) is cofinal for \(\mathcal B \times \omega\) and let \(U \in \mathscr N_{C_{\mathcal B}(X)}(\mathbf 0)\) be arbitrary.
We can find \(B \in \mathcal B\) and \(n \in \omega\) so that
\[
    [\mathbf 0; B, 2^{-n}] \subseteq U.
\]
By the cofinality of \(\mathcal F\) relative to \(\mathcal B \times \omega\), we can find \(A \in \mathcal A\) and \(m \in \omega\) so that \(B \subseteq A\) and \(n \leq m\).
It follows that
\[
  \psi(A, m) = [\mathbf 0; A , 2^{-m}] \subseteq  [\mathbf 0; B, 2^{-n}] \subseteq U.
\]
That is, \(\psi[\mathcal F]\) is cofinal in \(\mathscr N_{C_{\mathcal B}(X)}(\mathbf 0)\).

Without loss of generality, suppose \(\mathcal A = \bar{\mathcal A}\).
To address
\[
(\mathscr N_{C_{\mathcal A}(X)}(\mathbf 0); \mathscr N_{C_{\mathcal{B}}(X)}(\mathbf 0), \supseteq) \geq_T (\mathcal A \times \omega; \mathcal B \times \omega, \subseteq),
\]
let \(\phi : \mathscr N_{C_{\mathcal{A}}(X)}(\mathbf 0) \to \mathcal A \times \omega\) be defined in the following way.
For any \(U \in \mathscr N_{C_{\mathcal A}(X)}(\mathbf 0)\), let \(A_U \in \mathcal A\) and \(\varepsilon_U > 0\) be so that
\[
    [\mathbf 0; A_U , \varepsilon_U] \subseteq U.
\]
Choose \(n_U \in \omega\) so that \(2^{-n_U} < \varepsilon_U\).
Then define \(\phi(U) = \langle A_U, n_U \rangle\).

Suppose \(\mathcal F \subseteq \mathscr N_{C_{\mathcal A}(X)}(\mathbf 0)\) is cofinal for \(\mathscr N_{C_{\mathcal B}(X)}(\mathbf 0)\).
To see that \(\phi[\mathcal F]\) is cofinal for \(\mathcal B \times \omega\), let \(B \in \mathcal B\) and \(n \in \omega\).
Then \([\mathbf 0; B, 2^{-n}] \in \mathscr N_{C_{\mathcal B}(X)}(\mathbf 0)\) which means there exists some \(U \in \mathcal F\) so that \(U \subseteq [\mathbf 0; B, 2^{-n}]\).
Moreover,
\[
    [\mathbf 0; A_U , 2^{-n_U}] \subseteq U \subseteq [\mathbf 0; B, 2^{-n}].
\]

Suppose toward contradiction that \(B \not\subseteq A_U\).
Then, for \(x \in B \setminus A_U\), we can find a continuous function \(f : X \to [0,1]\) so that \(f(x) = 1\) and \(f \restriction_{A_U} \equiv 0\).
But then \(f \in [\mathbf 0; A_U , 2^{-n_U}] \setminus [\mathbf 0; B, 2^{-n}]\), a contradiction.

Were \(n > n_U\), consider the constant function defined by \(f(x) = 2^{-n}\).
This is a contradiction to \([\mathbf 0; A_U , 2^{-n_U}] \subseteq [\mathbf 0; B, 2^{-n}]\) so \(n \leq n_U\).

Since \(B \subseteq A_U\) and \(n \leq n_U\), we see that \(\phi[\mathcal F]\) is cofinal for \(\mathcal B \times \omega\).
\end{proof}

\section{Strengthening Strategies}

\begin{lemma}
    Suppose \(\mathcal A\) is an ideal-base, \(X = \bigcup \mathcal A\), and let \(\mathscr U \in \mathcal O(X,\mathcal A)\).
    Then, for each \(A\in \mathcal A\), \(\{U \in \mathscr U : A \in U \}\) is infinite.
    That is, \(\mathcal O(X,\mathcal A) = \Lambda(X,\mathcal A)\).
\end{lemma}
\begin{proof}
    Let \(A \in \mathcal A\) be arbitrary and let \(U_0 \in \mathscr U\) be so that \(A \subseteq U_0\).
    Since \(X \setminus U_0 \neq \emptyset\), let \(x_1 \in X \setminus U_0\) and let \(A_1^\ast \in \mathcal A\) be so that \(x_1 \in A_1^\ast\).
    Let \(A_1 \in \mathcal A\) be so that \(A \cup A_1^\ast \subseteq A_1\) and let \(U_1 \in \mathscr U\) be so that \(A_1 \subseteq U_1\).
    Since \(A_1 \cap (X \setminus U_0) \neq\emptyset\), we know that \(U_0 \neq U_1\).
    Inductively continue in this way.
\end{proof}

\begin{corollary}\label{lemma:Open=Large}
    Suppose \(\mathcal A\) and \(\mathcal B\) are ideal-bases.
    Then \(G_1(\mathscr N[\mathcal A], \neg\mathcal O(X,\mathcal B))\) is equivalent to \(G_1(\mathscr N[\mathcal A], \neg\Lambda(X,\mathcal B))\).
\end{corollary}

\begin{definition}
    For collections \(\mathcal A\) and \(\mathcal B\), recall that \(\mathcal A\) \textbf{refines} \(\mathcal B\), denoted \(\mathcal A \prec \mathcal B\), provided that, for every \(B \in \mathcal B\), there exists \(A \in \mathcal A\) so that \(A \subseteq B\).
\end{definition}

\begin{lemma}
\(\mathcal A \prec \mathcal B\) if and only if \(\mathcal O(X,\mathcal B) \subseteq \mathcal O(X,\mathcal A)\).
\end{lemma}
\begin{proof}
    Suppose \(\mathcal A \prec \mathcal B\).
    Let \(\mathscr U \in \mathcal O(X,\mathcal B)\) and \(A \in \mathcal A\).
    Let \(B \in \mathcal B\) be so that \(A \subseteq B\) and let \(U \in \mathscr U\) be so that \(B \subseteq U\).
    You get the idea.

    Now, suppose \(\mathcal A \not\prec \mathcal B\).
    Let \(A \in \mathcal A\) be so that, for all \(B\in\mathcal B\), \(A \not\subseteq B\).
    Then choose \(x_B \in A \setminus B\) and set \(U_B = X \setminus \{x_B\}\) for each \(B\in \mathcal B\).
    Notice that \(B \subseteq U_B\) so \(\{ U_B : B \in \mathcal B \} \in \mathcal O(X,\mathcal B)\).
    Clearly, \(\{ U_B : B \in \mathcal B \} \not\in \mathcal O(X,\mathcal A)\).
\end{proof}

In \cite{Pawlikowski1994}, Pawlikowski showed that \(S_{\text{fin}}(\mathcal O_X,\mathcal O_X)\) if and only if \(\One \not\uparrow G_{\text{fin}}(\mathcal O_X,\Lambda_X)\) and also that \(S_{1}(\mathcal O_X,\mathcal O_X)\) if and only if \(\One \not\uparrow G_{1}(\mathcal O_X,\Lambda_X)\).
The authors generalized this in a previous paper.
The following lemmas are slightly more general than proved there, but the proofs are the same as in \cite{CaruvanaHolshouser}.

{\color{red}
Lemmas \ref{lemma:PawlikowskiA} and \ref{lemma:Pawlikowski} are only known to be true if both cover types are \(\omega\)-covers or if both cover types are \(k\)-covers. See \href{https://arxiv.org/abs/2102.00296}{arXiv:2102.00296} for a proof of the \(k\)-covers case.

\begin{lemma} \label{lemma:PawlikowskiA}
    Assume \(\mathcal A \prec \mathcal B\) and \(S_{\text{fin}}(\mathcal O(X,\mathcal A),\mathcal O(X,\mathcal B))\).
    Then \(\One \not\uparrow G_{\text{fin}}(\mathcal O(X,\mathcal A),\Lambda(X,\mathcal B))\).
    Moreover, \(\One \uparrow G_{\text{fin}}(\mathcal O(X, \mathcal A), \mathcal O(X, \mathcal B))\) if and only if \(\One \underset{\text{pre}}{\uparrow} G_{\text{fin}}(\mathcal O(X, \mathcal A), \mathcal O(X, \mathcal B))\).
\end{lemma}

\begin{lemma}\label{lemma:Pawlikowski}
    Assume \(\mathcal A \prec \mathcal B\) and \(S_{1}(\mathcal O(X,\mathcal A),\mathcal O(X,\mathcal B))\).
    Then \(\One \not\uparrow G_{1}(\mathcal O(X,\mathcal A),\Lambda(X,\mathcal B))\).
    Moreover, \(\One \uparrow G_{1}(\mathcal O(X, \mathcal A), \mathcal O(X, \mathcal B))\) if and only if \(\One \underset{\text{pre}}{\uparrow} G_{1}(\mathcal O(X, \mathcal A), \mathcal O(X, \mathcal B))\).
\end{lemma}
}

In \cite{TkachukFE}, Tkachuk showed that \(\One \uparrow G_1([X]^{<\omega}, \neg\mathcal O_X)\) if and only if \(\One \uparrow G_1([X]^{<\omega},\neg\Gamma_X)\).
The authors generalized this result to \(\mathcal O(X,\mathcal A)\) in \cite{CaruvanaHolshouser}, assuming that \(\mathcal A\) is an ideal.
Here we show that one only needs to assume that \(\mathcal A\) is an ideal base.

\begin{lemma} \label{lem:PreviousLemma}
    For any strategy \(\sigma\) for One in \(G_1(\mathcal A, \mathcal B)\) where \(\mathcal A\) and \(\mathcal B\) are collections, define
    \[
        \text{play}_\sigma = \left\{ \langle x_0 , x_1 , \ldots , x_n \rangle : (n \in \omega) \wedge (\forall \ell < n)\left[ x_\ell \in \sigma(\langle x_j : j < \ell \rangle) \right] \right\} \subseteq \left(\bigcup \mathcal A\right)^{<\omega}
    \]
    and
    \[
        \text{play}_\sigma^\omega = \left\{  \langle x_n : n \in \omega \rangle : (\forall n \in \omega)\left[ \langle x_\ell : \ell \leq n \rangle \in \text{play}_\sigma \right] \right\} \subseteq \left( \bigcup \mathcal A \right)^\omega
    \]
    If \(\sigma\) is a winning strategy, then for any \(\langle x_n : n \in \omega \rangle \in \text{play}_\sigma^\omega\), \(\{ x_n : n \in \omega \} \not\in \mathcal B\).
\end{lemma}
\begin{proof}
    Let \(\langle x_n : n \in \omega \rangle \in \text{play}_\sigma^\omega\).
    Let \(A_0 = \sigma(\emptyset)\) and notice that \(x_0 \in A_0\) since \(\langle x_0 \rangle \in \text{play}_\sigma\).
    Now suppose we have \(A_0 , A_1 , \ldots , A_n \in \mathcal A\) defined so that \(x_\ell \in A_\ell = \sigma(\langle x_j : j < \ell \rangle)\).
    Let \(A_{n+1} = \sigma(\langle x_0,x_1, \ldots , x_n \rangle )\).
    We claim that \(x_{n+1} \in A_{n+1}\).
    To see this, we know that \(\langle x_0 , x_1 , \ldots , x_{n+1} \rangle \in \text{play}_\sigma\) so \(x_{n+1} \in \sigma(\langle x_j : j < n+1 \rangle) = A_{n+1}\).
    Hence, the \(x_n\) arise from a single run of the game according to \(\sigma\).

    Since \(\sigma\) is winning for One, \(\{x_n:n\in\omega\} \not\in\mathcal B\).
\end{proof}

\begin{proposition}
    Let \(\mathcal A\) and \(\mathcal B\) be collections.
    Set
    \[
    \mathcal B_\Gamma = \{B \in \mathcal B : (\mbox{for all infinite } B' \subseteq B)[B' \in \mathcal B]\}
    \]
    If \(\mathcal A\) is a filter base, then \(\One \uparrow G_1(\mathcal A, \neg \mathcal B)\) if and only if \(\One \uparrow G_1(\mathcal A, \neg \mathcal B_\Gamma)\).
\end{proposition}
\begin{proof}
    Let \(s\) be a winning strategy for One in \(G_1(\mathcal A, \neg\mathcal B)\).
    For \(\langle x_0,\cdots,x_n \rangle \in \mbox{play}_s\), define \(\gamma(x_0,\cdots,x_n) \in \mathcal A\) to be so that
    \[
    \gamma(x_0,\cdots,x_n) \subseteq \bigcap_{j=0}^n s(x_0,\cdots,x_j).
    \]

    Now we will define a winning strategy \(\sigma\) for One in \(G_1(\mathcal A, \neg\mathcal B_\Gamma)\).
    First set \(\sigma(\emptyset) = s(\emptyset) = A_0\).
    Now suppose we have defined \(\sigma(x_0,\cdots,x_{n-1})\) for all \(x_0,\cdots,x_{n-1}\) satisfying \(x_0 \in \sigma(\emptyset)\), \(x_1 \in \sigma(x_0)\), and so on.
    Suppose also that \(\sigma\) has been defined in such a way that for a fixed \(x_n \in \sigma(x_0,\cdots,x_{n-1})\),
    \begin{enumerate}[label=(\roman*)]
        \item for any \(0 \leq j_0 < j_1 < \cdots < j_k \leq n\), \(\langle x_{j_0}, x_{j_1}, \cdots, x_{j_k} \rangle \in \mbox{play}_s\), and
        \item for any \(0 \leq j_0 < j_1 < \cdots < j_k \leq \ell < n\), \(x_{\ell+1} \in \gamma(x_{j_0},x_{j_1},\cdots,x_{j_k})\).
    \end{enumerate}
    Define \(A_{n+1} \in \mathcal A\) to be so that
    \[
    A_{n+1} \subseteq \bigcap \{\gamma(x_{j_0},x_{j_1},\cdots,x_{j_k}) : 0 \leq j_0 < j_1 < \cdots < j_k \leq n\}
    \]
    Then set \(\sigma(x_0,\cdots,x_n) = A_{n+1}\).

    We check that this definition satisfies the two properties relative to \(n+1\).
    Fix \(x_{n+1} \in A_{n+1}\).
    Let \(0 \leq j_0 < j_1 < \cdots < j_k \leq n+1\).
    Notice that \(\langle x_{j_0},x_{j_1},\cdots,x_{j_{k-1}} \rangle \in \mbox{play}_s\) by the inductive hypothesis.
    So let \(A^*_{j_m} = s(x_{j_0},\cdots,x_{j_m})\) for \(0 \leq m < k\) and
    \[
    A^*_{j_k} = s(x_{j_0},x_{j_1},\cdots,x_{j_{k-1}}).
    \]
    It follows that \(A_{n+1} \subseteq A^*_{j_k}\) and that \(x_{n+1} \in A^*_{j_k}\).
    Hence,
    \[
    A^*_{j_0}, x_{j_0}, \cdots, A^*_{j_k}, x_{j_k}
    \]
    is a play according to \(s\).

    The second property holds by the definition of \(\sigma\).
    This completes the definition of \(\sigma\).

    We now show that \(\sigma\) is a winning strategy.
    Suppose \(A_0, x_0, A_1, x_1, \cdots\) is a full run of the game \(G_1(\mathcal A, \neg\mathcal B_\Gamma)\) played according to \(\sigma\).
    Suppose, by way of contradiction, that there is an infinite \(B' \subseteq \{x_n : n \in \omega\}\) so that \(B' \notin \mathcal B\).
    Say \(B' = \{x_{j_n} : n \in \omega\}\).
    Then by the construction of \(\sigma\), \(\langle x_0, \cdots, x_{j_n} \rangle \in \mbox{play}_s\) for all \(n \in \omega\).
    Hence, \(\{x_{j_n} : n \in \omega\} \in \mbox{play}^\omega_s\), and so by the Lemma \ref{lem:PreviousLemma}, \(\{x_{j_n} : n \in \omega\} = B' \in \mathcal B\), a contradiction.
    Thus \(\{x_n : \in \omega\} \in \mathcal B_\Gamma\), and \(\sigma\) is a winning strategy.

    The other direction of the proof is obvious.
\end{proof}

\begin{corollary}
    \label{lem:Open=Gamma}
	Let \(\mathcal A\) be an ideal-base.
	Then One has winning (pre-determined) strategy for the game  \(G_1(\mathscr N[\mathcal A], \neg\mathcal O(X,\mathcal B))\) if and only if One has winning (pre-determined) strategy for \(G_1(\mathscr N[\mathcal A],\neg\Gamma(X,\mathcal B))\).
	The same is true for pre-determined strategies.
\end{corollary}
\begin{proof}
    Notice that if \(\mathcal A\) is an ideal base, then \(\mathscr N[\mathcal A]\) is a filter base.
    Also notice that \(\mathcal O(X,\mathcal B)_\Gamma\) is the same thing as \(\Gamma(X, \mathcal B)\).
    This shows that \(\One \uparrow G_1(\mathscr N[\mathcal A], \neg\mathcal O(X,\mathcal B)) \iff \One \uparrow G_1(\mathscr N[\mathcal A],\neg\Gamma(X,\mathcal B))\).

    The fact that the results hold for pre-determined strategies follows from a modification of the proof of the proposition.
    Simply set
    \[
    \sigma(n) = s(0) \cap \cdots \cap s(n)
    \]
    and check that this works.
\end{proof}

\section{An Order on Single Selection Games}

\begin{definition}
    Let \(\mathcal A\), \(\mathcal B\), \(\mathcal C\), and \(\mathcal D\) be collections and \(\alpha\) be an ordinal.
    Say that \(G^\alpha_1(\mathcal A, \mathcal C) \leq_{\Two} G^\alpha_1(\mathcal B, \mathcal D)\) if
    \begin{itemize}
        \item \(\Two \underset{\text{mark}}{\uparrow} G^\alpha_1(\mathcal A, \mathcal C) \implies \Two \underset{\text{mark}}{\uparrow} G^\alpha_1(\mathcal B, \mathcal D)\),
        \item \(\Two \uparrow G^\alpha_1(\mathcal A, \mathcal C) \implies \Two \uparrow G^\alpha_1(\mathcal B, \mathcal D)\),
        \item \(\One \not\uparrow G^\alpha_1(\mathcal A, \mathcal C) \implies \One \not\uparrow G^\alpha_1(\mathcal B, \mathcal D)\), and
        \item \(\One \not\underset{\text{pre}}{\uparrow} G^\alpha_1(\mathcal A, \mathcal C) \implies \One \not\underset{\text{pre}}{\uparrow} G^\alpha_1(\mathcal B, \mathcal D)\).
    \end{itemize}
\end{definition}

Notice that if \(G^\alpha_1(\mathcal A, \mathcal C) \leq_{\Two} G^\alpha_1(\mathcal B, \mathcal D)\) and \(G^\alpha_1(\mathcal B, \mathcal D) \leq_{\Two} G^\alpha_1(\mathcal A, \mathcal C)\), then the games are equivalent.
Also notice that \(\leq_{\Two}\) is transitive.

\begin{theorem}\label{Translation}
    Let \(\mathcal A\), \(\mathcal B\), \(\mathcal C\), and \(\mathcal D\) be collections and \(\alpha\) be an ordinal.
    Suppose there are functions
    \begin{itemize}
        \item \(\overleftarrow{T}_{\One,\xi}:\mathcal B \to \mathcal A\) and
        \item \(\overrightarrow{T}_{\Two,\xi}: \bigcup \mathcal A \times \mathcal B \to \bigcup \mathcal B\)
    \end{itemize}
    for each \(\xi \in \alpha\), so that
    \begin{enumerate}[label=(Tr\arabic*)]
        \item \label{TranslationA} If \(x \in \overleftarrow{T}_{\One,\xi}(B)\), then \(\overrightarrow{T}_{\Two,\xi}(x,B) \in B\)
        \item \label{TranslationB} If \(x_\xi \in \overleftarrow{T}_{\One,\xi}(B_\xi)\) and \(\{x_\xi : \xi \in \alpha\} \in \mathcal C\), then \(\{\overrightarrow{T}_{\Two,\xi}(x_\xi,B_\xi) : \xi \in \alpha\} \in \mathcal D\).
    \end{enumerate}
    Then \(G^\alpha_1(\mathcal A,\mathcal C) \leq_{\Two} G^\alpha_1(\mathcal B, \mathcal D)\).
\end{theorem}
\begin{proof}
    Suppose \(\Two \underset{\text{mark}}{\uparrow} G^\alpha_1(\mathcal A, \mathcal C)\) and let \(\tau\) be a winning Markov strategy for Two.
    We define a winning Markov strategy for Two in \(G^\alpha_1(\mathcal B, \mathcal D)\).
    Toward this end, let \(\{B_\xi : \xi \in \alpha\} \subseteq \mathcal B\) be arbitrary and set \(A_\xi = \overleftarrow{T}_{\One,\xi}(B_\xi)\) and \(x_\xi = \tau(A_\xi , \xi)\).
    Define \(y_\xi = \overrightarrow{T}_{\Two,\xi}(x_\xi,B_\xi)\).
    Then
    \[
        \{ x_\xi : \xi \in \alpha\} \in \mathcal C \implies \{ y_\xi : \xi \in \alpha \} \in \mathcal D.
    \]

    Suppose \(\Two \uparrow G^\alpha_1(\mathcal A, \mathcal C)\) and let \(\tau\) be a winning strategy for Two.
    We define a strategy \(t\) for Two in \(G_1^\alpha(\mathcal B, \mathcal D)\) recursively.
    Suppose One plays \(B_0\).
    Then \(A_0 := \overleftarrow{T}_{\One,0}(B_0)\) is an initial play of \(G^\alpha_1(\mathcal A, \mathcal C)\).
    So \(x_0 := \tau(A_0) \in A_0\).
    Define
    \[
        t(B_0) = y_0 = \overrightarrow{T}_{\Two,0}(x_0,B_0).
    \]
    For \(\beta \in \alpha\), suppose we have \(\{A_\xi: \xi < \beta\}\), \(\{B_\xi : \xi < \beta \}\), \(\{ x_\xi : \xi < \beta\}\), and \(\{ y_\xi : \xi < \beta\}\) defined.
    Given \(B_\beta \in \mathcal B\), let \(A_\beta = \overleftarrow{T}_{\One,\beta}(B_\beta)\) and \(x_\beta = \tau(A_0, \ldots , A_\beta) \in A_\beta\).
    Then set
    \[
        t(B_0 , \ldots , B_\beta) = y_\beta = \overrightarrow{T}_{\Two,\beta}(x_\beta,B_\beta).
    \]
    This concludes the definition of \(t\).
    By \ref{TranslationA}, since \(x_\xi \in \overleftarrow{T}_{\One,\xi}(B_\xi)\), it follows that \(y_\xi \in B_\xi\).
    Using \ref{TranslationB}, we see that
    \[
        \{ x_\xi : \xi \in \alpha \} \in \mathcal C \implies \{ y_\xi : \xi \in \alpha \} \in \mathcal D.
    \]

    Suppose \(\One \uparrow G^\alpha_1(\mathcal B, \mathcal D)\) and let \(\sigma\) witness this.
    We will develop a strategy \(s\) for One in \(G^\alpha_1(\mathcal A, \mathcal B)\).
    Let \(B_0 = \sigma(\emptyset)\) and \(s(\emptyset) = A_0 = \overleftarrow{T}_{\One,0}(B_0)\).
    Then, for \(\beta \in \alpha\), suppose we have \(\{A_\xi:\xi \leq \beta\} \subseteq \mathcal A\), \(\{ B_\xi: \xi \leq \beta \} \subseteq \mathcal B\), \(\{x_\xi: \xi < \beta\}\), and \(\{y_\xi:\xi < \beta\}\) defined in the right way.
    Suppose \(x_\beta \in A_\beta\).
    Then set \(y_\beta = \overrightarrow{T}_{\Two,\beta}(x_\beta,B_\beta) \in B_\beta\), \(B_{\beta+1} = \sigma( y_0 , \ldots , y_\beta)\) and
    \[
        s( x_0 , \ldots , x_\beta) = A_{\beta+1} = \overleftarrow{T}_{\One,\beta+1}(B_{\beta+1}).
    \]
    After the run of the game is completed, let \(x_{\xi+1} \in s(x_0,\cdots,x_\xi)\) for all \(\xi \in \alpha\) and \(x_0 \in s(\emptyset)\).
    Then \ref{TranslationA} gives us that \(\overrightarrow{T}_{\Two,\xi}(x_\xi,B_\xi) = y_\xi \in B_\xi\).
    As \(\sigma\) is a winning strategy for One in \(\One \uparrow G^\alpha_1(\mathcal B, \mathcal D)\), \ref{TranslationB} yields
    \[
        \{ y_\xi : \xi \in \alpha \} \not\in \mathcal D \implies \{ x_\xi : \xi \in \alpha \} \not\in \mathcal C
    \]

    Suppose \(\One \underset{\text{pre}}{\uparrow} G^\alpha_1(\mathcal B, \mathcal D)\) and let \(\{B_\xi : \xi \in \alpha \}\) represent One's winning strategy.
    Let \(A_\xi = \overleftarrow{T}_{\One,\xi}(B_\xi)\) for each \(\xi\in\alpha\).
    We will show that \(\{A_\xi:\xi\in\alpha\}\) forms a winning strategy for One in \(G^\alpha_1(\mathcal A, \mathcal C)\).
    Let \(x_\xi \in A_\xi\) for all \(\xi \in \alpha\) and let \(y_\xi = \overrightarrow{T}_{\Two,\xi}(x_\xi,B_\xi)\).
    By \ref{TranslationA}, \(y_\xi \in B_\xi\) for all \(\xi\in\alpha\) and so \(\{y_\xi:\xi\in\alpha\} \not\in \mathcal D\).
    By \ref{TranslationB}, we see that \(\{x_\xi : \xi \in \alpha\} \not\in \mathcal C\).
\end{proof}

In some situations, the use of both maps is not necessary as the translation between player One's moves simply comes from lifting the translation of player Two's selections.

\begin{corollary}\label{corollary:EasyTranslate}
    Let \(\mathcal A\), \(\mathcal B\), \(\mathcal C\), and \(\mathcal D\) be collections.
    Suppose there is a map \(\phi : \left( \bigcup \mathcal B \right) \times \omega \to \left( \bigcup \mathcal A \right)\) so that
    \begin{itemize}
        \item
        For all \(B \in \mathcal B\) and all \(n \in \omega\), \(\{ \phi(y,n) : y \in B\} \in \mathcal A\)
        \item
        if \(\{ \phi(y_n,n) : n \in \omega \} \in \mathcal C\), then \(\{ y_n : n \in \omega \} \in \mathcal D\)
    \end{itemize}
    Then \(G_1(\mathcal A, \mathcal C) \leq_\Two G_1(\mathcal B, \mathcal D)\).
\end{corollary}
\begin{proof}
    Define \(\overleftarrow{T}_{\One,n}:\mathcal B \to \mathcal A\) by
    \[
    \overleftarrow{T}_{\One,n}(B) = \phi[B \times \{n\}].
    \]
    From the first assumption on \(\phi\) we know that \(\overleftarrow{T}_{\One,n}\) really does produce objects in \(\mathcal A\).
    For \(x \in \phi[B \times \{n\}]\) and \(n \in \omega\), choose \(y_{x,n} \in B\) so that \(\phi(y_{x,n},n) = x\).
    Define \(\overrightarrow{T}_{\Two,n}:\bigcup \mathcal A \times \mathcal B \to \bigcup \mathcal B\) by
    \[
    \overrightarrow{T}_{\Two,n}(x)(B) = y_{x,n}
    \]
    if possible and otherwise set it to be an arbitrary element of \(\bigcup \mathcal B\).
    So if \(x \in \overleftarrow{T}_{\One,n}(B)\), then \(\overrightarrow{T}_{\Two,n}(x)(B) = y_{x,n} \in B\).

    Now suppose \(x_n \in \phi[B_n \times \{n\}]\) and \(\{x_n : n \in \omega\} \in \mathcal C\).
    Then \(\{\phi(y_{x_n,n},n) : n \in \omega\} \in \mathcal C\).
    By the second assumption on \(\phi\), it follows that \(\{y_{x_n} : n \in \omega\} \in \mathcal D\).
    Thus \(\{\overrightarrow{T}_{\Two,n}(x_n)(B_n) : n \in \omega\} \in \mathcal D\).
    This completes the proof.
\end{proof}

\section{Equivalent and Dual Classes of Games}

\begin{corollary}\label{corollary:Roth=CFT=CDFT}
Let \(X\) be a Tychonoff space and \(\mathcal A, \mathcal B \subseteq \wp(X)\).
Then
\begin{enumerate}[label=(\roman*)]
    \item \label{RothA} \(G_1(\mathcal O(X,\mathcal A), \Lambda(X, \mathcal B)) \leq_{\Two} G_1(\Omega_{C_{\mathcal A}(X), \mathbf{0}}, \Omega_{C_{\mathcal B}(X),\mathbf{0}})\),
    \item \label{RothB} \(G_1(\Omega_{C_{\mathcal A}(X), \mathbf{0}}, \Omega_{C_{\mathcal B}(X),\mathbf{0}}) \leq_{\Two} G_1(\mathcal D_{C_{\mathcal A}(X)}, \Omega_{C_{\mathcal B}(X),\mathbf{0}})\), and
    \item \label{RothC} if \(\mathcal A\) consists of closed sets and \(X\) is \(\mathcal A\)-normal, then
    \[
    G_1(\mathcal D_{C_{\mathcal A}(X)}, \Omega_{C_{\mathcal B}(X),\mathbf{0}}) \leq_{\Two} G_1(\mathcal O(X,\mathcal A), \Lambda(X, \mathcal B)).
    \]
\end{enumerate}
Thus if \(\mathcal A\) consists of closed sets and \(X\) is \(\mathcal A\)-normal, then the three games are equivalent.
\end{corollary}
\begin{proof}
    Let \(\phi : C_{\mathcal A}(X) \times \omega \to \mathscr T_X\) be defined by \(\phi(f,n) = f^{-1}[(-2^{-n} , 2^{-n})]\).
    Suppose \(F \in \Omega_{C_{\mathcal A}(X) , \mathbf 0}\) and let both \(A \in \mathcal A\) and \(n \in \omega\) be arbitrary.
    Choose \(f \in F\) so that \(f \in [\mathbf 0; A , 2^{-n}]\) and notice that \(A \subseteq f^{-1}[(-2^{-n},2^{-n})]\).
    Hence, \(\{ \phi(f,n) : f \in F\} \in \mathcal O(X,\mathcal A)\).

    Next, suppose \(\{ \phi(f_n , n) : n \in \omega\} \in \Lambda(X,\mathcal B)\).
    Let \(B \in \mathcal B\) and \(\varepsilon > 0\) be arbitrary.
    Then, there is \(n \in \omega\) large enough so that \(B \subseteq f_n^{-1}[(-2^{-n} , 2^{-n})]\) and \(2^{-n} < \varepsilon\).
    It follows that \(f \in [\mathbf 0 ; B , \varepsilon]\).
    By Corollary \ref{corollary:EasyTranslate}, this completes \ref{RothA}.

    Next we check that \(G_1(\Omega_{C_{\mathcal A}(X), \mathbf{0}}, \Omega_{C_{\mathcal B}(X),\mathbf{0}}) \leq_{\Two} G_1(\mathcal D_{C_{\mathcal A}(X)}, \Omega_{C_{\mathcal B}(X),\mathbf{0}})\).
    As \(\mathcal D_{C_{\mathcal A}(X)} \subseteq \Omega_{C_{\mathcal A}(X), \mathbf{0}}\), this is true.
    Simply have Two use the exact same counter-play or strategy.

    For \ref{RothC}, define
    \begin{itemize}
        \item \(\overleftarrow{T}_{\One,n}:\mathcal O(X,\mathcal A) \to \mathcal D_{C_{\mathcal A}(X)}\) by
        \[
        \overleftarrow{T}_{\One,n}(\mathcal U) = \{f \in C_{\mathcal A}(X) : (\exists U \in \mathcal U)[f[X \smallsetminus U] = 1]\}
        \]
        \item \(\overrightarrow{T}_{\Two,n}:C_{\mathcal A}(X) \times \mathcal O(X,\mathcal A) \to \mathscr{T}_X\) by \(\overrightarrow{T}_{\Two,n}(f,\mathcal U) = U\), where \(U \in \mathcal U\) is such that \(f[X \smallsetminus U] = \{1\}\) (if possible, otherwise set \(\overrightarrow{T}_{\Two,n}(f,\mathcal U) = X\)).
    \end{itemize}

    First check that the functions are well-defined.
    To see that \(\overleftarrow{T}_{\One,n}(\mathcal U)\) is a dense set in \(C_{\mathcal A}(X)\), consider a basic open set \([f; A, \varepsilon]\).
    Since \(\mathcal U \in \mathcal O(X, \mathcal A)\), there is a \(U \in \mathcal U\) so that \(A \subseteq U\).
    Since \(X\) is \(\mathcal A\)-normal, we can find a continuous function \(g:X \to [0,1]\) so that \(g[A] = 0\) and \(g[X \smallsetminus A] = 1\).
    Define \(h = f(1-g) + g\).
    Then \(h\restriction_A = f\), \(h[X \smallsetminus U] = 1\).
    So \(h \in [f;A,\varepsilon] \cap \overleftarrow{T}_{\One,n}(\mathcal U, n)\).
    This shows that \(\overleftarrow{T}_{\One,n}(\mathcal U, n)\) is dense.
    It is clear that \(\overrightarrow{T}_{\Two,n}\) maps into the appropriate space.

    We next check \ref{TranslationA}.
    Suppose \(f \in \overleftarrow{T}_{\One,n}(\mathcal U)\).
    We need to check that \(\overrightarrow{T}_{\Two,n}(f,\mathcal U) \in \mathcal U\).
    Because \(f \in \overleftarrow{T}_{\One,n}(\mathcal U)\), we can find a \(U \in \mathcal U\) so that \(f[X \smallsetminus U] = \{1\}\).
    Thus \(\overrightarrow{T}_{\Two,n}(f,\mathcal U) = U \in \mathcal U\).

    Now we check \ref{TranslationB}, that is, that the \(\overrightarrow{T}_{\Two,n}\) translate from \(\Omega_{C_{\mathcal B}(X),\mathbf{0}}\) to \(\Lambda(X,\mathcal B)\).
    Suppose \(f_n \in \overleftarrow{T}_{\One,n}(\mathcal U_n)\) and
    \[
    \{f_n : n \in \omega\} \in \Omega_{C_{\mathcal B}(X),\mathbf{0}}.
    \]
    We need to see that \(\{\overrightarrow{T}_{\Two,n}(f_n, \mathcal U_n) : n \in \omega\} \in \Lambda(X,\mathcal B)\).
    Notice \(\overrightarrow{T}_{\Two,n}(f_n,\mathcal U_n) = U_n \in \mathcal U_n\) with the property that \(f_n[X \smallsetminus U_n] = 1\).
    Let \(B \in \mathcal B\).
    Then there is an \(n_0\) so that \(f_{n_0} \in [\mathbf 0;B,1]\).
    Thus \(B \subseteq f_{n_0}^{-1}[(-1,1)]\), and so \(B \cap (X \setminus U_{n_0}) = \emptyset\).
    Therefore \(B \subseteq U_{n_0}\).
    There is an \(n_1 > n_0\) so that \(f_{n_1} \in [\mathbf 0;B,1] \smallsetminus \{f_k : k \leq n_0\}\) and so \(B \subseteq U_{n_1}\).
    Continuing this process inductively, we see that \(B\) is covered infinitely many times and that \(\{U_n : n \in \omega\} \in \Lambda(X,\mathcal B)\).
\end{proof}

\begin{corollary}\label{corollary:PO=Gru=CD}
Let \(X\) be a Tychonoff space and \(\mathcal A, \mathcal B \subseteq \wp(X)\).
Then
\begin{enumerate}[label=(\roman*)]
    \item \label{POA} \(G_1(\mathscr{N}_{C_{\mathcal{A}(X)}}(\mathbf 0), \neg \Omega_{C_{\mathcal B(X)},\mathbf 0}) \leq_{\Two} G_1(\mathscr{N}[\mathcal A], \neg \Lambda(X,\mathcal B))\)
    \item \label{POB} \(G_1(\mathscr{T}_{C_{\mathcal{A}(X)}}, \neg \Omega_{C_{\mathcal B(X)},\mathbf 0}) \leq_{\Two} G_1(\mathscr{N}_{C_{\mathcal{A}(X)}}(\mathbf 0), \neg \Omega_{C_{\mathcal B(X)},\mathbf 0})\)
    \item \label{POC} \(G_1(\mathscr{T}_{C_{\mathcal A}(X)}, \mbox{CD}_{C_{\mathcal B}(X)}) \leq_{\Two} G_1(\mathscr{T}_{C_{\mathcal{A}(X)}}, \neg \Omega_{C_{\mathcal B(X)},\mathbf 0})\)
    \item \label{POD} If \(\mathcal A\) consists of closed sets, \(X\) is \(\mathcal A\)-normal, and \(\mathcal B\) consists of \(\mathbb R\)-bounded sets, then
    \[
    G_1(\mathscr{N}[\mathcal A], \neg \Lambda(X,\mathcal B)) \leq_{\Two} G_1(\mathscr{T}_{C_{\mathcal A}(X)}, \mbox{CD}_{C_{\mathcal B}(X)}).
    \]
\end{enumerate}
Thus if \(\mathcal A\) consists of closed sets, \(X\) is \(\mathcal A\)-normal, and \(\mathcal B\) consists of \(\mathbb R\)-bounded sets, then all these games are equivalent.
\end{corollary}
\begin{proof}
First we check that \(G_1(\mathscr{N}_{C_{\mathcal{A}(X)}}(\mathbf 0), \neg \Omega_{C_{\mathcal B(X)},\mathbf 0}) \leq_{\Two} G_1(\mathscr{N}[\mathcal A], \neg \Lambda(X,\mathcal B))\).
Define
\begin{itemize}
    \item \(\overleftarrow{T}_{\One,n}:\mathscr{N}[\mathcal A] \to \mathscr{N}_{C_{\mathcal{A}(X)}}(\mathbf 0)\) by
    \[
    \overleftarrow{T}_{\One,n}(\mathscr{N}(A)) = [\mathbf 0; A, 2^{-n}]
    \]
    \item \(\overrightarrow{T}_{\Two,n}:C_{\mathcal A}(X) \times \mathscr{N}[\mathcal A] \to \mathscr{T}_X\) by \(\overrightarrow{T}_{\Two,n}(f,\mathscr{N}(A)) = f^{-1}[(-2^{-n},2^{-n})]\).
\end{itemize}

The maps are well-defined since the continuous pre-image of an open set is open.

We check \ref{TranslationA}.
Suppose \(f \in \overleftarrow{T}_{\One,n}(\mathscr{N}(A))\).
We need to check that \(\overrightarrow{T}_{\Two,n}(f,\mathscr{N}(A)) \in \mathscr{N}(A)\), i.e. that \(A \subseteq f^{-1}[(-2^{-n},2^{-n})]\).
Since \(f \in \overleftarrow{T}_{\One,n}(\mathscr{N}(A)) = [\mathbf 0; A, 2^{-n}]\), \(f[A] \subseteq (-2^{-n},2^{-n})\).
Thus \(A \subseteq f^{-1}[(-2^{-n},2^{-n})]\).

We check \ref{TranslationB}.
Suppose \(f_n \in \overleftarrow{T}_{\One,n}(\mathscr{N}(A_n))\) and that \(\{f_n : n \in \omega\} \notin \Omega_{C_{\mathcal B(X)},\mathbf 0}\).
Then \(f_n \in [\mathbf 0; A_n, 2^{-n}]\) and there is a \(B \in \mathcal B\), an \(\varepsilon > 0\), and an \(N \in \omega\) so that for all \(n \geq N\), \(f_n \notin [\mathbf 0; B, \varepsilon]\).
We need to show that \(\{f_n^{-1}[(-2^{-n},2^{-n})] : n \in \omega\} \notin \Lambda(X,\mathcal B)\).
We proceed by way of contradiction.
Suppose in particular that there is a \(n \geq N\) so that \(2^{-n} < \varepsilon\) and \(B \subseteq f_n^{-1}[(-2^{-n},2^{-n})]\).
Then \(f_n \in [\mathbf 0; B, 2^{-n}] \subseteq [\mathbf 0; B, \varepsilon]\).
This is a contradiction.

Then \(G_1(\mathscr{T}_{C_{\mathcal{A}(X)}}(\mathbf 0), \neg \Omega_{C_{\mathcal B(X)},\mathbf 0}) \leq_{\Two} G_1(\mathscr{N}_{C_{\mathcal{A}(X)}}(\mathbf 0), \neg \Omega_{C_{\mathcal B(X)},\mathbf 0})\) is true as \(\mathscr{N}_{C_{\mathcal{A}(X)}}(\mathbf 0) \subseteq \mathscr{T}_{C_{\mathcal{A}(X)}}\).

To see that \(G_1(\mathscr{T}_{C_{\mathcal A}(X)}, \mbox{CD}_{C_{\mathcal B}(X)}) \leq_{\Two} G_1(\mathscr{T}_{C_{\mathcal{A}(X)}}, \neg \Omega_{C_{\mathcal B(X)},\mathbf 0})\), observe that if Two can create a closed discrete set in response to player One, then Two has avoided having \(\mathbf 0\) as a cluster point.

Suppose \(X\) is \(\mathcal A\)-normal and \(\mathcal B\) consists of \(\mathbb R\)-bounded sets.
For \(U \in \mathscr{T}_{C_{\mathcal A}(X)}\), \(V \in \mathscr N(A_U)\), and \(n \in \omega\), identify a function \(f_{U,V,n}:X \to \mathbb R\) with the property that \(f_{U,V,n}\restriction_{A_U} = f_U\) and \(f_{U,V,n}[X \smallsetminus V] = \{n\}\).
Such a function exists for the following reason.
Since \(X\) is \(\mathcal A\)-normal, there is a function \(g\) so that \(g[A_U] = 0\) and \(g[X \smallsetminus V] = 1\).
Let \(f_{U,V,n} = f_U \cdot (1 - g) + n \cdot g\) and notice that \(f_{U,V,n}\) is as required.
Define
\begin{itemize}
    \item \(\overleftarrow{T}_{\One,n}:\mathscr{T}_{C_{\mathcal A}(X)} \to \mathscr{N}[\mathcal A]\) by \(\overleftarrow{T}_{\One,n}(U) = \mathscr N(A_U)\)
    \item \(\overrightarrow{T}_{\Two,n}:\mathscr{T}_X \times \mathscr{T}_{C_{\mathcal A}(X)} \to C_{\mathcal A}(X)\) by \(\overrightarrow{T}_{\Two,n}(V,U) = f_{U,V,n}\) (if possible, otherwise declare \(\overrightarrow{T}_{\Two,n}(V,U) = \mathbf 0\)).
\end{itemize}

We check \ref{TranslationA}.
Suppose \(V \in \overleftarrow{T}_{\One,n}(U) = \mathscr N(A_U)\).
We need to check that \(\overrightarrow{T}_{\Two,n}(V,U) = f_{U,V,n} \in U\).
Since \(V \in \mathscr N(A_U)\), \(f_{U,V,n}\) was chosen so that \(f_{U,V,n}\restriction_{A_U} = f_U\) which implies that \(f_{U,V,n} \in U\).

We check \ref{TranslationB}.
Suppose \(V_n \in \overleftarrow{T}_{\One,n}(U_n) = \mathscr N(A_n)\), where \(A_n = A_{U_n}\) and \(\{V_n : n \in \omega\} \notin \Lambda (X,\mathcal B)\).
Then there is a \(B \in \mathcal B\) and \(N\) so that for all \(n \geq N\), \(B \not \subseteq V_n\).
Say \(\overrightarrow{T}_{\Two,n}(V_n,U_n) = g_n\) and that \(f_{U_n} = f_n\).
Then \(g_n\restriction_{A_n} = f_n\) and \(g_n[X \smallsetminus V_n] = \{n\}\).
We proceed by way of contradiction.
Let \(f \in C_{\mathcal B}(X)\) be so that for all \(n\), there is a \(k \geq \max\{N,n\}\) so that \(g_k \in [f;B,2^{-n}]\).
Since \(B \not\subseteq V_k\), there is an \(x_k \in B \smallsetminus V_k\).
Thus \(|g_k(x_n) - f(x_n)| \leq 2^{-n}\), and so \(f(x_n) \geq k-1\).
Proceeding in this way, we can produce an unbounded sequence \(k_n\) and a collection of points \(x_n \in B\) so that \(f(x_n) \geq k_n - 1\).
But then \(f\) is a continuous function where \(f[B]\) is unbounded.
So \(B\) is not \(\mathbb R\)-bounded, which is a contradiction.
\end{proof}

\begin{corollary}\label{corollary:Gru<PointGamma}
Let \(X\) be a Tychonoff space and \(\mathcal A, \mathcal B \subseteq \wp(X)\).
Then
\begin{enumerate}[label=(\roman*)]
    \item \label{gruA} \(G_1(\mathscr{N}_{C_{\mathcal{A}(X)}}(\mathbf 0), \neg \Gamma_{C_{\mathcal B(X)},\mathbf 0}) \leq_{\Two} G_1(\mathscr{N}[\mathcal A], \neg \Gamma(X,\mathcal B))\)
    \item \label{gruB} \(G_1(\mathscr{N}_{C_{\mathcal{A}(X)}}(\mathbf 0), \neg \Omega_{C_{\mathcal B(X)},\mathbf 0}) \leq_{\Two} G_1(\mathscr{N}_{C_{\mathcal{A}(X)}}(\mathbf 0), \neg \Gamma_{C_{\mathcal B(X)},\mathbf 0})\)
\end{enumerate}
\end{corollary}
\begin{proof}
Part \ref{gruA} of this corollary is essentially the same as \ref{POA} of Corollary \ref{corollary:PO=Gru=CD}.

To see that \(G_1(\mathscr{N}_{C_{\mathcal{A}(X)}}(\mathbf 0), \neg \Omega_{C_{\mathcal B(X)},\mathbf 0}) \leq_{\Two} G_1(\mathscr{N}_{C_{\mathcal{A}(X)}}(\mathbf 0), \neg \Gamma_{C_{\mathcal B(X)},\mathbf 0})\), simply notice that if Two can avoid clustering around \(\mathbf 0\), then Two can certainly avoid converging to \(\mathbf 0\).
\end{proof}

\begin{definition}
    For a collection \(\mathcal A\), we say that \(\mathcal B \subseteq \mathcal A\) is a \textbf{selection basis} for \(\mathcal A\) if
    \[
        (\forall A \in \mathcal A)(\exists B \in \mathcal B)(B \subseteq A).
    \]
\end{definition}

\begin{definition}
    For collections \(\mathcal A\) and \(\mathcal R\), we say that \(\mathcal R\) is a \textbf{reflection} of \(\mathcal A\) if
    \[
        \{ \text{ran}(f) : f \in \text{choice}(\mathcal R) \}
    \]
    is a selection basis for \(\mathcal A\).
\end{definition}

\begin{theorem}\cite[Corollary 17]{ClontzDuality} \label{thm:ClontzDuality}
    If \(\mathcal R\) is a reflection of \(\mathcal A\), then \(G_1(\mathcal A, \mathcal B)\) and \(G_1(\mathcal R, \neg\mathcal B)\) are dual.
\end{theorem}

\begin{corollary}\cite[Corollary 21]{CaruvanaHolshouser}
    For any collection \(\mathcal A\) of subsets of a space \(X\) and any collection \(\mathcal B\), the games \(G_1(\mathcal O(X,\mathcal A),\mathcal B)\) and \(G_1(\mathscr N[\mathcal A] , \neg\mathcal B)\) are dual.
\end{corollary}

\begin{proposition}
    Suppose \(X\) is a topological space, \(x \in X\), and \(\mathcal B \subseteq \wp(X)\).
    Then \(G_1(\Omega_{X,x}, \mathcal B)\) and \(G_1(\mathscr N(x), \neg \mathcal B)\) are dual.
\end{proposition}
\begin{proof}
    It suffices to show that
    \[
    \{\mbox{ran}(C) : C \in \mbox{choice}(\mathscr N(x))\} \subseteq \Omega_{X,x}
    \]
    and is a selection basis for \(\Omega_{X,x}\).
    Clearly, each \(\mbox{ran}(C) \in \Omega_{X,x}\).
    Now let \(F \in \Omega_{X,x}\).
    Then for each \(U \in \mathscr N(x)\), there is an \(x_U \in F \cap U\).
    Define a choice function \(C\) for \(\mathscr N(x)\) by \(C(U) = x_U\).
    Notice that
    \[
    \mbox{ran}(C) = \{x_U : U \in \mathscr N(x)\} \subseteq F.
    \]
\end{proof}

\begin{corollary}\label{corollary:CFTDual}
    \(G_1(\Omega_{C_{\mathcal A}(X),\mathbf 0}, \Omega_{C_{\mathcal B}(X),\mathbf 0})\)
    and \(G_1(\mathscr N_{C_{\mathcal A}(X)}(\mathbf 0), \neg \Omega_{C_{\mathcal B}(X),\mathbf 0})\) are dual.
\end{corollary}

\begin{proposition}\label{prop:CD->CDFT}
    \(G_1(\mathcal D_{C_{\mathcal A}(X)}, \Omega_{C_{\mathcal B}(X),\mathbf 0})\) and \(G_1(\mathscr{T}_{C_{\mathcal A}(X)}, \neg \Omega_{C_{\mathcal B}(X),\mathbf 0})\) are dual.
    Therefore whenever \(\Two \uparrow G_1(\mathscr{T}_{C_{\mathcal A}(X)}, \mbox{CD}_{C_{\mathcal B}(X)})\), we have that \(\One \uparrow G_1(\mathcal D_{C_{\mathcal A}(X)}, \Omega_{C_{\mathcal B}(X),\mathbf 0})\). This is also true for going from Markov strategies to pre-determined strategies.
\end{proposition}
\begin{proof}
    We can use reflection to show that \(G_1(\mathcal D_{C_{\mathcal A}(X)}, \Omega_{C_{\mathcal B}(X),\mathbf 0})\)
    and \(G_1(\mathscr{T}_{C_{\mathcal A}(X)}, \neg\Omega_{C_{\mathcal B}(X),\mathbf 0})\) are dual and that \(G_1(\mathcal D_{C_{\mathcal A}(X)}, \neg \mbox{CD}_{C_{\mathcal B}(X)})\) and \(G_1(\mathscr{T}_{C_{\mathcal A}(X)}, \mbox{CD}_{C_{\mathcal B}(X)})\) are dual as well.
    First check that
    \[
    \{\mbox{ran}(C) : C \in \mbox{choice}(\mathscr{T}_{C_{\mathcal A}(X)})\} \subseteq \mathcal D_{C_{\mathcal A}(X)}
    \]
    and is a selection basis for \(\mathcal D_{C_{\mathcal A}(X)}\).
    Clearly, each \(\mbox{ran}(C) \in \mathcal D_{C_{\mathcal A}(X)}\).
    Now let \(D \in \mathcal D_{C_{\mathcal A}(X)}\).
    Then for each \(U \in \mathscr{T}_{C_{\mathcal A}(X)}\), there is an \(f_U \in D \cap U\).
    Define a choice function \(C\) for \(\mathscr{T}_{C_{\mathcal A}(X)}\) by \(C(U) = f_U\).
    Notice that
    \[
    \mbox{ran}(C) = \{f_U : U \in \mathscr{T}_{C_{\mathcal A}(X)}\} \subseteq D.
    \]
    Thus \(G_1(\mathcal D_{C_{\mathcal A}(X)}, \mathcal C)\)
    and \(G_1(\mathscr{T}_{C_{\mathcal A}(X)}, \neg \mathcal C)\) are dual for any \(\mathcal C \subseteq \wp(C_{\mathcal B}(X))\).

    Therefore,
    \[
    \Two \uparrow G_1(\mathscr{T}_{C_{\mathcal A}(X)}, \mbox{CD}_{C_{\mathcal B}(X)})
    \implies \Two \uparrow G_1(\mathscr{T}_{C_{\mathcal A}(X)}, \neg\Omega_{C_{\mathcal B}(X),\mathbf 0})
    \iff \One \uparrow G_1(\mathcal D_{C_{\mathcal A}(X)}, \Omega_{C_{\mathcal B}(X),\mathbf 0}).
    \]
    The analogous results hold for Markov and pre-determined strategies.
\end{proof}

\section{Covering Properties}

\begin{lemma}\label{lemma:Gru=Cof}
    Suppose \(X\) is a Tychonoff space. Then the following are equivalent:
    \begin{enumerate}[label=(\roman*)]
        \item \label{OnePre_Gruenhage}
        \(\One \underset{\text{pre}}{\uparrow} G_1( \mathscr N_{C_{\mathcal A}(X)}(\mathbf 0) , \neg \Gamma_{C_{\mathcal B}(X), \mathbf 0})\),
        \item \label{OnePre_GruenhageCluster}
        \(\One \underset{\text{pre}}{\uparrow} G_1( \mathscr N_{C_{\mathcal A}(X)}(\mathbf 0) , \neg \Omega_{C_{\mathcal B}(X), \mathbf 0})\),
        \item \label{GruenhageCofinal}
        \(\cof(\mathscr N_{C_{\mathcal A}(X)}(\mathbf 0); \mathscr N_{C_{\mathcal{B}}(X)}(\mathbf 0), \supseteq) = \omega\).
    \end{enumerate}
\end{lemma}
\begin{proof}
    Clearly, \ref{OnePre_Gruenhage} implies \ref{OnePre_GruenhageCluster}.

    Suppose \(\One \underset{\text{pre}}{\uparrow} G_1( \mathscr N_{C_{\mathcal A}(X)}(\mathbf 0) , \neg \Omega_{C_{\mathcal B}(X), \mathbf 0})\).
    Then we get a sequence of neighborhoods \([\mathbf 0; A_n, \varepsilon_n]\).
    Now, let \(B \in \mathcal B\), \(\varepsilon > 0\), and consider \([\mathbf 0; B, \varepsilon]\).
    Suppose \([\mathbf 0; A_n, \varepsilon_n] \not \subseteq [\mathbf 0; B, \varepsilon]\) for any \(n\).
    Then we have functions \(f_n \in [\mathbf 0; A_n, \varepsilon_n] \setminus [\mathbf 0; B, \varepsilon]\).
    Consider the play \([\mathbf 0; A_0, \varepsilon_0], f_0, \cdots\) of the game \(G_1(\mathscr{T}_{C_{\mathcal A}(X)}, \neg \Omega_{C_{\mathcal B}(X), \mathbf 0})\) according to the winning strategy.
    Because none of the \(f_n\) are in \([\mathbf 0; B, \varepsilon]\), the \(f_n\) fail to accumulate to \(\mathbf 0\) in \(C_{\mathcal B}(X)\).
    This is a contradiction.
    So \ref{OnePre_GruenhageCluster} implies \ref{GruenhageCofinal}.

    Now let \(U_n\) be a sequence of \(C_{\mathcal A}(X)\) neighborhoods of \(\mathbf 0\) which is cofinal in the \(C_{\mathcal B}(X)\) neighborhoods.
    We can assume without loss of generality that the \(U_n\) are descending.
    Define a strategy \(\sigma\) for player One in \(G_1( \mathscr N_{C_{\mathcal A}(X)}(\mathbf 0) , \neg \Gamma_{C_{\mathcal B}(X), \mathbf 0})\) by \(\sigma(n) = U_n\).
    Suppose that \(f_n \in U_n\) for all \(n\).
    Let \([\mathbf 0; B, \varepsilon]\) be an arbitrary \(C_{\mathcal B}(X)\)-nhood of \(\mathbf 0\).
    Then there is an \(N\) so that for all \(n \geq N\), \(U_n \subseteq [\mathbf 0; B, \varepsilon]\).
    Thus for all \(n \geq N\), \(f_n \in [\mathbf 0; B, \varepsilon]\).
    So \(f_n \to \mathbf 0\) in \(C_{\mathcal B}(X)\).
    Therefore \ref{GruenhageCofinal} implies \ref{OnePre_Gruenhage}.
\end{proof}

The following generalizes V.416 from \cite[p. 460]{TkachukFE}.
Moreover, if we replace \(\mathcal A\) with the space of singletons \(X\), we obtain Theorem 1 of Gerlits and Nagy, \cite{GerlitsNagy}.

\begin{lemma} \label{lemma:PreDeterminedClosed}
     Assume \(\mathcal A, \mathcal B \subseteq \wp(X)\).
     Then \(\One \underset{\text{pre}}{\uparrow} G_1(\mathscr N[\mathcal A], \neg\mathcal O(X, \mathcal B))\) if and only if \(\cof(\mathcal A; \mathcal B, \subseteq) \leq \omega\).
\end{lemma}
\begin{proof}
    Suppose \(\One \underset{\text{pre}}{\uparrow} G_1(\mathscr N[\mathcal A], \neg \mathcal O(X, \mathcal B))\).
    Let \(\sigma\) be an example of a pre-determined strategy for One in this game.
    Say \(\mbox{ran}(\sigma) = \{\mathscr N(A_n) : n \in \omega\}\).
    We claim that \(\{A_n : n \in \omega\}\) is cofinal for \(\mathcal B\).
    Towards a contradiction suppose that there were an \(B \in \mathcal B\) so that \(B \not \subseteq A_n\) for all \(n\).
    Then for each \(n\), we can choose \(x_n \in B \setminus A_n\).
    Then the sequence \(\mathscr N(A_0), X \setminus \{x_0\}, \cdots\) would be a play of \(G_1(\mathscr N[\mathcal A], \neg\mathcal O(X, \mathcal B))\).
    Since \(\sigma\) is winning, \(B \subseteq X \setminus \{x_n\}\) for some \(n\).
    But then \(x_n \in X \setminus \{x_n\}\), a contradiction.
    Therefore \(\{A_n : n \in \omega\}\) is cofinal for \(\mathcal B\) and \(\cof(\mathcal A; \mathcal B, \subseteq) \leq \omega\).

    Suppose \(\cof(\mathcal A; \mathcal B, \subseteq) = \omega\).
    Let \(\{A_n : n \in \omega\}\) witness this.
    Define a strategy \(\sigma\) by \(\sigma(n) = \mathscr N(A_n)\).
    Now suppose \(\sigma(0), U_0, \cdots\) is a play of \(G_1(\mathscr N[\mathcal A], \neg \mathcal O(X, \mathcal B))\) according to \(\sigma\).
    Let \(B \in \mathcal B\).
    Then there is an \(n\) so that \(B \subseteq A_n \subseteq U_n\).
    Thus \(\{U_n : n \in \omega\} \in \mathcal O(X, \mathcal B)\) and \(\sigma\) is winning.
\end{proof}

The following generalizes a result of Telg{\'{a}}rsky \cite{Telgarsky} and extends Theorem 27 of \cite{CaruvanaHolshouser}.

\begin{lemma} \label{lemma:CoveringPropertyPre}
     Assume \(\mathcal A, \mathcal B \subseteq \wp(X)\), and \(\mathcal A\) is a collection of \(G_\delta\) sets.
     Then the following are equivalent:
     \begin{enumerate}[label=(\roman*)]
         \item \(\One \uparrow G_1(\mathscr N[\mathcal A], \neg\mathcal O(X, \mathcal B))\)
         \item \(\cof(\mathcal A; \mathcal B, \subseteq) \leq \omega\)
         \item \(\One \underset{\text{pre}}{\uparrow} G_1(\mathscr N[\mathcal A], \neg\mathcal O(X, \mathcal B))\)
    \end{enumerate}
\end{lemma}
\begin{proof}
    Let \(\sigma\) be a strategy for One.
    Without loss of generality, One is playing sets from \(\mathcal A\) and Two plays open sets which contain One's play.
    For every \(A \in \mathcal A\), let \(\mathcal U_A\) be a countable collection of open sets so that \(A = \bigcap \mathcal U_A\).

    Define a tree in the following way.
    Let \(T_0 = \emptyset\).
    For \(n \in \omega\), we define
    \[
        T_{n+1} = \{ w \concat \langle \sigma(w), U \rangle : w\in T_n \text{ and } U \in \mathcal U_{\sigma(w)} \}.
    \]
    Observe that each \(T_n\) is countable as each \(\mathcal U_A\) is countable.
    Hence,
    \[
        \mathscr F := \bigcup_{n\in\omega} \{ \sigma(w) : w \in T_n \}
    \]
    is a countable subset of \(\mathcal A\).

    By way of contradiction, suppose there is some \(B \in \mathcal B\) so that \(B \not\subseteq A\) for all \(A \in \mathscr F\).
    Let \(A_0 = \sigma(\emptyset)\).
    Since \(B \not\subseteq A_0\), there must be some \(x_0 \in B\setminus A_0\).
    As \(A_0 = \bigcap \mathcal U_{A_0}\), there is some \(U_0 \in \mathcal U_{A_0}\) so that \(A_0 \subseteq U_0\) and \(x_0 \not\in U_0\).

    Recursively, this defines a run of the game \(A_0, U_0, A_1 , U_1 , \ldots\) according to \(\sigma\).
    So we can conclude that \(\{ U_n : n \in \omega \} \in \mathcal O(X,\mathcal B)\).
    Thus, \(B \subseteq U_n\) for some \(n \in \omega\) but then \(x_n \in U_n\), a contradiction.
    Therefore, \(\cof(\mathcal A; \mathcal B, \subseteq) \leq \omega\).

    The rest of the equivalence is clear.
\end{proof}

\begin{note}
    Let \(X\) be the one-point Lindel\"{o}fication of \(\omega_1\) and consider \(G_1(\mathscr N[[X]^{<\omega}], \neg \mathcal O(X, [X]^{<\omega}))\).
    In \(X\), \(\{\omega_1\}\) is closed, but not a \(G_\delta\).
    One has a winning strategy in \(G_1(\mathscr N[[X]^{<\omega}], \neg \mathcal O(X, [X]^{<\omega}))\), but \(\cof([X]^{<\omega}; [X]^{<\omega} , \subseteq) = \omega_1\).

    Now consider \(X = \mathbb{R}\).
    Let \(\mathcal M\) be the meager subsets of \(\mathbb{R}\).
    Then player One has a winning tactic (in two moves) for \(G_1(\mathscr N[\mathcal M], \neg \mathcal O_X)\), but \(\cof(\mathcal M; X, \subseteq) = \mbox{cov}(\mathcal M) > \omega\).
\end{note}

\section{The Main Theorems}
\begin{theorem}\label{MD1}
Suppose \(X\) is a Tychonoff space and \(\mathcal A, \mathcal B \subseteq \wp(X)\).
Suppose \(\mathcal A\) and \(\mathcal B\) are ideal-bases and that \(\mathcal A\) consists of closed sets.
Then the following diagrams are true, where dashed arrows require the assumption that \(X\) is \(\mathcal A\)-normal and dotted lines require the assumption that \(\mathcal B\) consists of \(\mathbb R\)-bounded sets.

If \(X\) is \(\mathcal A\)-normal, \(\mathcal B\) consists of \(\mathbb R\)-bounded sets, and \(\mathcal A\) consists of \(G_\delta\) sets, then all of the statements across both diagrams are equivalent.

\noindent
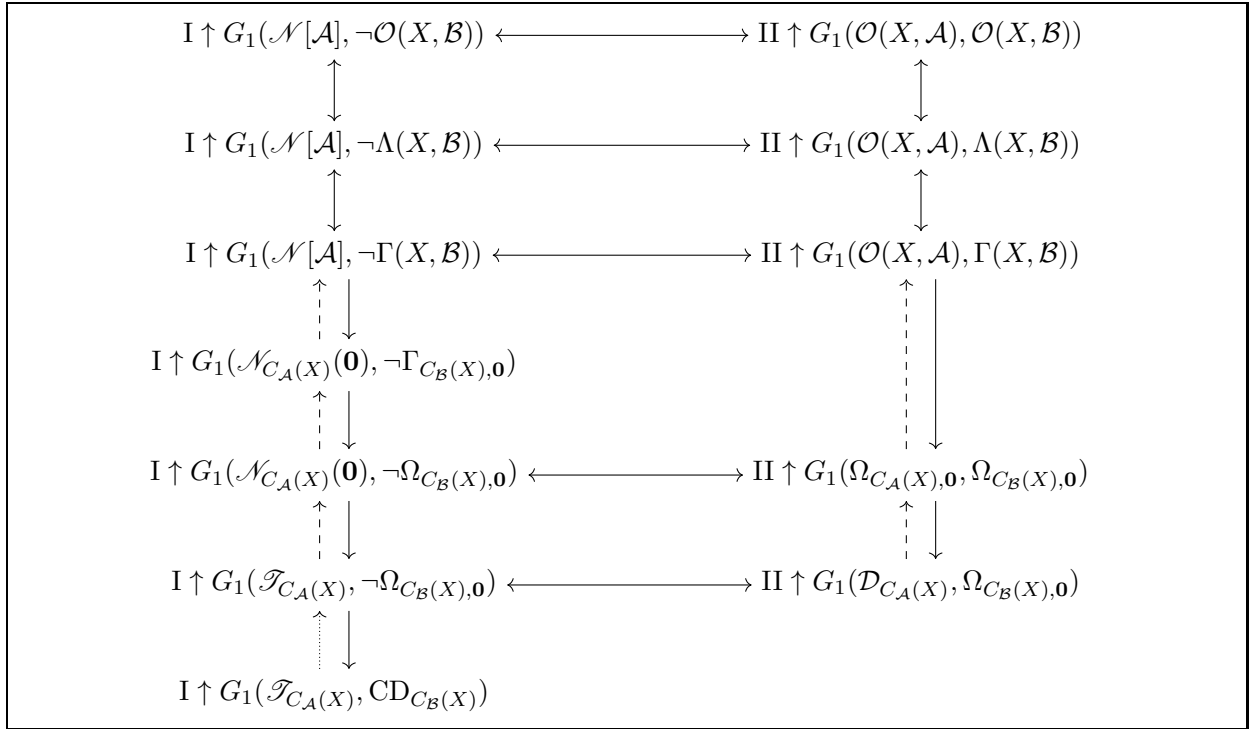
\framebox[\textwidth]{
\begin{tikzpicture}[scale=.975, auto, transform shape]
    \node (pointOpen) at (-4,0) {\(\One \uparrow G_1(\mathscr N[\mathcal A], \neg \mathcal O(X,\mathcal B))\)};
    \node (pointLambda) at (-4,-1.5) {\(\One \uparrow G_1(\mathscr N[\mathcal A], \neg \Lambda(X,\mathcal B))\)};
    \node (pointGamma) at (-4,-3) {\(\One \uparrow G_1(\mathscr N[\mathcal A], \neg \Gamma(X,\mathcal B))\)};

    \node (GruenhageLim) at (-4,-4.5) {\(\One \uparrow G_1(\mathscr N_{C_{\mathcal A}(X)}(\mathbf 0), \neg \Gamma_{C_{\mathcal B}(X),\mathbf 0})\)};
    \node (GruenhageCluster) at (-4,-6) {\(\One \uparrow G_1(\mathscr N_{C_{\mathcal A}(X)}(\mathbf 0), \neg \Omega_{C_{\mathcal B}(X),\mathbf 0})\)};
    \node (CL) at (-4,-7.5) {\(\One \uparrow G_1(\mathscr{T}_{C_{\mathcal A}(X)},\neg \Omega_{C_{\mathcal B}(X),\mathbf 0})\)};
    \node (CD) at (-4,-9) {\(\One \uparrow G_1(\mathscr{T}_{C_{\mathcal A}(X)},\mbox{CD}_{C_{\mathcal B}(X)})\)};

    \node (Rothberger) at (4,0) {\(\Two \uparrow G_1(\mathcal O(X, \mathcal A), \mathcal O(X,\mathcal B))\)};
    \node (RothbergerLambda) at (4,-1.5) {\(\Two \uparrow G_1(\mathcal O(X, \mathcal A), \Lambda(X,\mathcal B))\)};
    \node (RothbergerGamma) at (4,-3) {\(\Two \uparrow G_1(\mathcal O(X, \mathcal A), \Gamma(X,\mathcal B))\)};

    \node (spacer) at (4,-4.5) {\phantom{\(\One \uparrow G_1(\mathscr N_{C_{\mathcal A}(X)}(\mathbf 0), \neg \Gamma_{C_{\mathcal B}(X),\mathbf 0})\)}};
    \node (CFT) at (4,-6) {\(\Two \uparrow G_1(\Omega_{C_{\mathcal A}(X),\mathbf 0},\Omega_{C_{\mathcal B}(X),\mathbf 0})\)};
    \node (CDFT) at (4,-7.5) {\(\Two \uparrow G_1(\mathcal D_{C_{\mathcal A}(X)},\Omega_{C_{\mathcal B}(X),\mathbf 0})\)};

    \draw[<->] (pointOpen.east) -- (Rothberger.west);

    \draw[<->] (pointOpen.south) -- (pointLambda.north);
    \draw[<->] (Rothberger.south) -- (RothbergerLambda.north);

    \draw[<->] (pointLambda.east) -- (RothbergerLambda.west);

    \draw[<->] (pointLambda.south) -- (pointGamma.north);
    \draw[<->] (RothbergerLambda.south) -- (RothbergerGamma.north);

    \draw[<->] (pointGamma.east) -- (RothbergerGamma.west);

    \draw[->] ([xshift=.2cm]pointGamma.south) -- ([xshift=.2cm]GruenhageLim.north);
    \draw[->, dashed] ([xshift=-.2cm]GruenhageLim.north) -- ([xshift=-.2cm]pointGamma.south);
    \draw[->] ([xshift=.2cm]RothbergerGamma.south) -- ([xshift=.2cm]CFT.north);
    \draw[->, dashed] ([xshift=-.2cm]CFT.north) -- ([xshift=-.2cm]RothbergerGamma.south);

    \draw[->] ([xshift=.2cm]GruenhageLim.south) -- ([xshift=.2cm]GruenhageCluster.north);
    \draw[->, dashed] ([xshift=-.2cm]GruenhageCluster.north) -- ([xshift=-.2cm]GruenhageLim.south);

    \draw[<->] (GruenhageCluster.east) -- (CFT.west);

    \draw[->] ([xshift=.2cm]GruenhageCluster.south) -- ([xshift=.2cm]CL.north);
    \draw[->, dashed] ([xshift=-.2cm]CL.north) -- ([xshift=-.2cm]GruenhageCluster.south);
    \draw[->] ([xshift=.2cm]CFT.south) -- ([xshift=.2cm]CDFT.north);
    \draw[->, dashed] ([xshift=-.2cm]CDFT.north) -- ([xshift=-.2cm]CFT.south);

    \draw[<->] (CL.east) -- (CDFT.west);

    \draw[->] ([xshift=.2cm]CL.south) -- ([xshift=.2cm]CD.north);
    \draw[->, densely dotted] ([xshift=-.2cm]CD.north) -- ([xshift=-.2cm]CL.south);
\end{tikzpicture}
}

\noindent
\framebox[\textwidth]
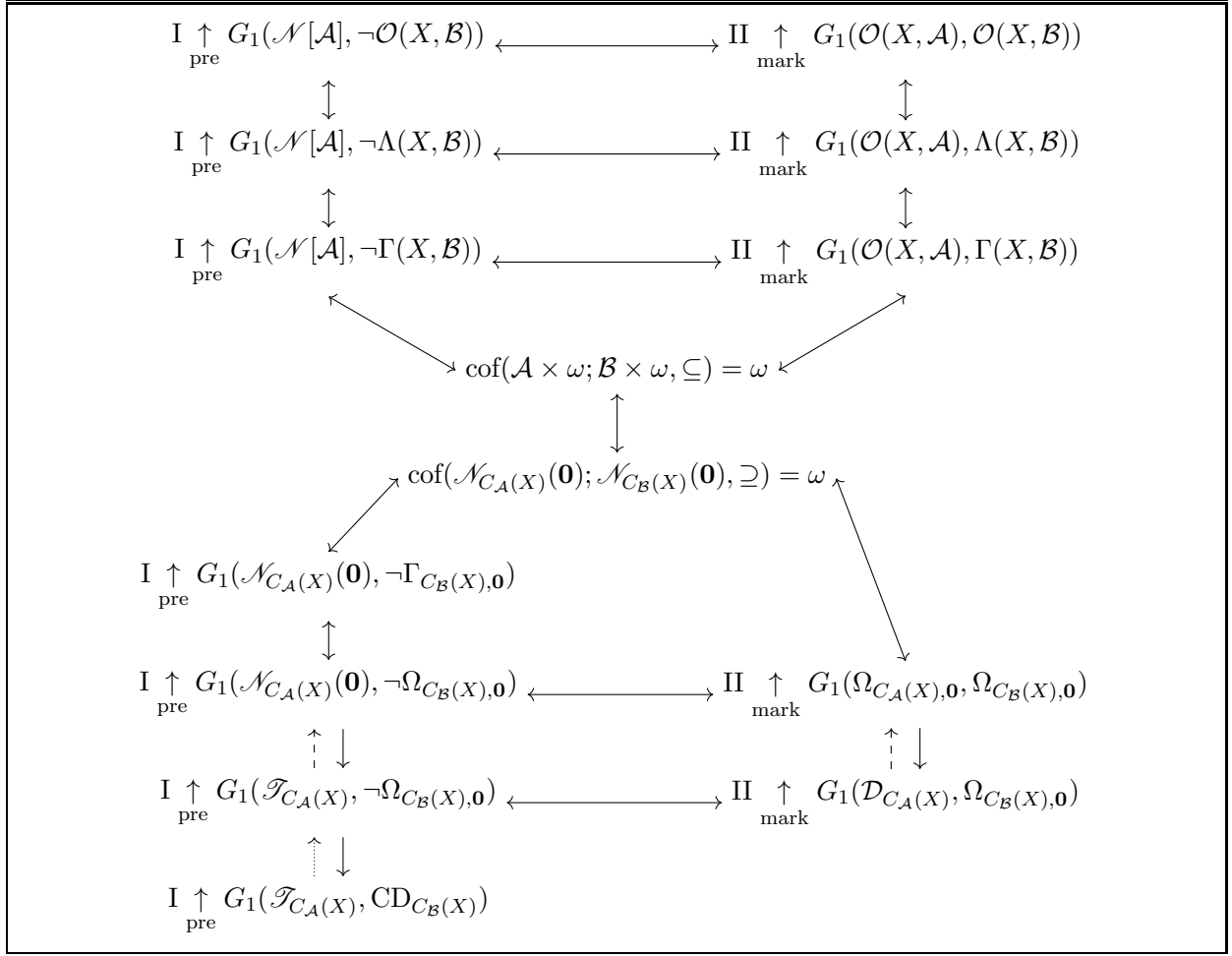
{
\begin{tikzpicture}[scale=.975, auto, transform shape]
    \node (XCof) at (0,0) {\(\cof(\mathcal A \times \omega; \mathcal B \times \omega, \subseteq) = \omega\)};

    \node (FunCof) at (0,-1.5) {\(\cof(\mathscr N_{C_{\mathcal A}(X)}(\mathbf 0); \mathscr N_{C_{\mathcal B}(X)}(\mathbf 0), \supseteq) = \omega\)};

    \node (pointGamma) at (-4,1.5) {\(\One \underset{\text{pre}}{\uparrow} G_1(\mathscr N[\mathcal A], \neg \Gamma(X,\mathcal B))\)};
    \node (pointLambda) at (-4,3) {\(\One \underset{\text{pre}}{\uparrow} G_1(\mathscr N[\mathcal A], \neg \Lambda(X,\mathcal B))\)};
    \node (pointOpen) at (-4,4.5) {\(\One \underset{\text{pre}}{\uparrow} G_1(\mathscr N[\mathcal A], \neg \mathcal O(X,\mathcal B))\)};

    \node (RothbergerGamma) at (4,1.5) {\(\Two \underset{\text{mark}}{\uparrow} G_1(\mathcal O(X, \mathcal A), \Gamma(X,\mathcal B))\)};
    \node (RothbergerLambda) at (4,3) {\(\Two \underset{\text{mark}}{\uparrow} G_1(\mathcal O(X, \mathcal A), \Lambda(X,\mathcal B))\)};
    \node (Rothberger) at (4,4.5) {\(\Two \underset{\text{mark}}{\uparrow} G_1(\mathcal O(X, \mathcal A), \mathcal O(X,\mathcal B))\)};

    \node (GruenhageLim) at (-4,-3) {\(\One \underset{\text{pre}}{\uparrow} G_1(\mathscr N_{C_{\mathcal A}(X)}(\mathbf 0), \neg \Gamma_{C_{\mathcal B}(X),\mathbf 0})\)};
    \node (GruenhageCluster) at (-4,-4.5) {\(\One \underset{\text{pre}}{\uparrow} G_1(\mathscr N_{C_{\mathcal A}(X)}(\mathbf 0), \neg \Omega_{C_{\mathcal B}(X),\mathbf 0})\)};
    \node (CL) at (-4,-6) {\(\One \underset{\text{pre}}{\uparrow} G_1(\mathscr{T}_{C_{\mathcal A}(X)},\neg \Omega_{C_{\mathcal B}(X),\mathbf 0})\)};
    \node (CD) at (-4,-7.5) {\(\One \underset{\text{pre}}{\uparrow} G_1(\mathscr{T}_{C_{\mathcal A}(X)},\mbox{CD}_{C_{\mathcal B}(X)})\)};

    \node (spacer) at (4,-3) {\phantom{\(\One \underset{\text{pre}}{\uparrow} G_1(\mathscr N_{C_{\mathcal A}(X)}(\mathbf 0), \neg \Gamma_{C_{\mathcal B}(X),\mathbf 0})\)}};
    \node (CFT) at (4,-4.5) {\(\Two \underset{\text{mark}}{\uparrow} G_1(\Omega_{C_{\mathcal A}(X),\mathbf 0},\Omega_{C_{\mathcal B}(X),\mathbf 0})\)};
    \node (CDFT) at (4,-6) {\(\Two \underset{\text{mark}}{\uparrow} G_1(\mathcal D_{C_{\mathcal A}(X)},\Omega_{C_{\mathcal B}(X),\mathbf 0})\)};

    \draw[<->] (pointOpen.east) -- (Rothberger.west);

    \draw[<->] (pointOpen.south) -- (pointLambda.north);
    \draw[<->] (Rothberger.south) -- (RothbergerLambda.north);

    \draw[<->] (pointLambda.east) -- (RothbergerLambda.west);

    \draw[<->] (pointLambda.south) -- (pointGamma.north);
    \draw[<->] (RothbergerLambda.south) -- (RothbergerGamma.north);

    \draw[<->] (pointGamma.east) -- (RothbergerGamma.west);

    \draw[<->] (pointGamma.south) -- (XCof.west);
    \draw[<->] (RothbergerGamma.south) -- (XCof.east);

    \draw[<->] (XCof.south) -- (FunCof.north);

    \draw[<->] (FunCof.west) -- (GruenhageLim.north);
    \draw[<->] (FunCof.east) -- (CFT.north);

    \draw[<->] (GruenhageLim.south) -- (GruenhageCluster.north);

    \draw[<->] (GruenhageCluster.east) -- (CFT.west);

    \draw[->] ([xshift=.2cm]GruenhageCluster.south) -- ([xshift=.2cm]CL.north);
    \draw[->, dashed] ([xshift=-.2cm]CL.north) -- ([xshift=-.2cm]GruenhageCluster.south);
    \draw[->] ([xshift=.2cm]CFT.south) -- ([xshift=.2cm]CDFT.north);
    \draw[->, dashed] ([xshift=-.2cm]CDFT.north) -- ([xshift=-.2cm]CFT.south);

    \draw[<->] (CL.east) -- (CDFT.west);

    \draw[->] ([xshift=.2cm]CL.south) -- ([xshift=.2cm]CD.north);
    \draw[->, densely dotted] ([xshift=-.2cm]CD.north) -- ([xshift=-.2cm]CL.south);

\end{tikzpicture}
}
\end{theorem}
\begin{proof}
    Since we have assumed that \(\mathcal A\) and \(\mathcal B\) are ideal-bases, Lemma \ref{lem:Open=Gamma} implies that all three versions of the generalized point-open game are equivalent for player One.
    This applies for full strategies and pre-determined strategies.

    The fact that \(\One \uparrow G_1(\mathscr N[\mathcal A], \neg \Psi(X,\mathcal B))\) is equivalent to \(\Two \uparrow G_1(\mathcal O(X, \mathcal A), \Psi(X,\mathcal B))\) (where \(\Psi\) is \(\mathcal O\), \(\Lambda\), or \(\Gamma\)) comes from the general reflection result from Clontz, Theorem \ref{thm:ClontzDuality}.
    This also implies the analogous statements for pre-determined and Markov strategies.
    Since all of the versions of the generalized point-open game are equivalent for player One, we can conclude that all of the versions of the generalized Rothberger game are equivalent for player Two.

    By Corollary \ref{corollary:CFTDual}, \(\Two \uparrow G_1(\Omega_{C_{\mathcal A}(X),\mathbf 0},\Omega_{C_{\mathcal B}(X),\mathbf 0})\) if and only if \(\One \uparrow G_1(\mathscr N_{C_{\mathcal A}(X)}(\mathbf 0), \neg \Omega_{C_{\mathcal B}(X),\mathbf 0})\), and also at the level of Markov/pre-determined strategies.
    Likewise, Proposition \ref{prop:CD->CDFT} implies that \(\Two \uparrow G_1(\mathcal D_{C_{\mathcal A}(X)},\Omega_{C_{\mathcal B}(X),\mathbf 0})\) is equivalent to \(\One \uparrow G_1(\mathscr{T}_{C_{\mathcal A}(X)}, \neg \Omega_{C_{\mathcal B}(X),\mathbf 0})\), and also at the level of Markov/pre-determined strategies.

    Corollary \ref{corollary:Roth=CFT=CDFT} yields the implications between \(G_1(\mathcal O(X, \mathcal A), \Lambda(X,\mathcal B))\), \(G_1(\Omega_{C_{\mathcal A}(X),\mathbf 0},\Omega_{C_{\mathcal B}(X),\mathbf 0})\), and \(G_1(\mathcal D_{C_{\mathcal A}(X)},\Omega_{C_{\mathcal B}(X),\mathbf 0})\).
    Then Corollaries \ref{corollary:PO=Gru=CD} and \ref{corollary:Gru<PointGamma} provide the arrows between games for the rest of the left side of the diagram.

    We now check the improved implications in the second diagram.
    By Lemma \ref{lemma:PreDeterminedClosed}, we have that \(\One \underset{\text{pre}}{\uparrow} G_1(\mathscr N[\mathcal A], \neg \mathcal O(X,\mathcal B))\) if and only if \(\cof(\mathcal A \times \omega, \mathcal B \times \omega) = \omega\).
    Then Lemma \ref{lemma:CofinalityBetweenGroundAndFunctions} implies that \(\cof(\mathcal A \times \omega, \mathcal B \times \omega) = \omega\) if and only if \(\cof(\mathscr N_{C_{\mathcal A}(X)}(\mathbf 0); \mathscr N_{C_{\mathcal B}(X)}(\mathbf 0), \supseteq) = \omega\).
    Finally, using Lemma \ref{lemma:Gru=Cof}  we see that \(\cof(\mathscr N_{C_{\mathcal A}(X)}(\mathbf 0); \mathscr N_{C_{\mathcal B}(X)}(\mathbf 0), \supseteq) = \omega\) if and only if \(\One \underset{\text{pre}}{\uparrow} G_1(\mathscr N_{C_{\mathcal A}(X)}(\mathbf 0), \neg \Omega_{C_{\mathcal B}(X),\mathbf 0})\), which is in turn equivalent to \(\One \underset{\text{pre}}{\uparrow} G_1(\mathscr N_{C_{\mathcal A}(X)}(\mathbf 0), \neg \Gamma_{C_{\mathcal B}(X),\mathbf 0})\).
    This suffices to improve the arrows from the first diagram and finishes the second diagram.

\end{proof}

\begin{theorem} \label{MD2}
Suppose \(X\) is a Tychonoff space and \(\mathcal A, \mathcal B \subseteq \wp(X)\).
Suppose \(\mathcal A\) and \(\mathcal B\) are ideal-bases and that \(\mathcal A\) consists of closed sets.
Then the following diagrams are true, where dashed arrows require the assumption that \(X\) is \(\mathcal A\)-normal and dotted lines require the assumption that \(X\) is \(\mathcal A\)-normal and \(\mathcal B\) consists of \(\mathbb R\)-bounded sets.

{\color{red}
If \(X\) is \(\mathcal A\)-normal, \(\mathcal B\) consists of \(\mathbb R\)-bounded sets, and \(\mathcal A \prec \mathcal B\), then all of the statements across both diagrams are equivalent.
}
Correction: If \(\mathcal A = \mathcal B\) are the finite subsets of \(X\), or \(\mathcal A = \mathcal B\) are the compact subsets of \(X\), then all of the statements across both diagrams are equivalent.
\noindent
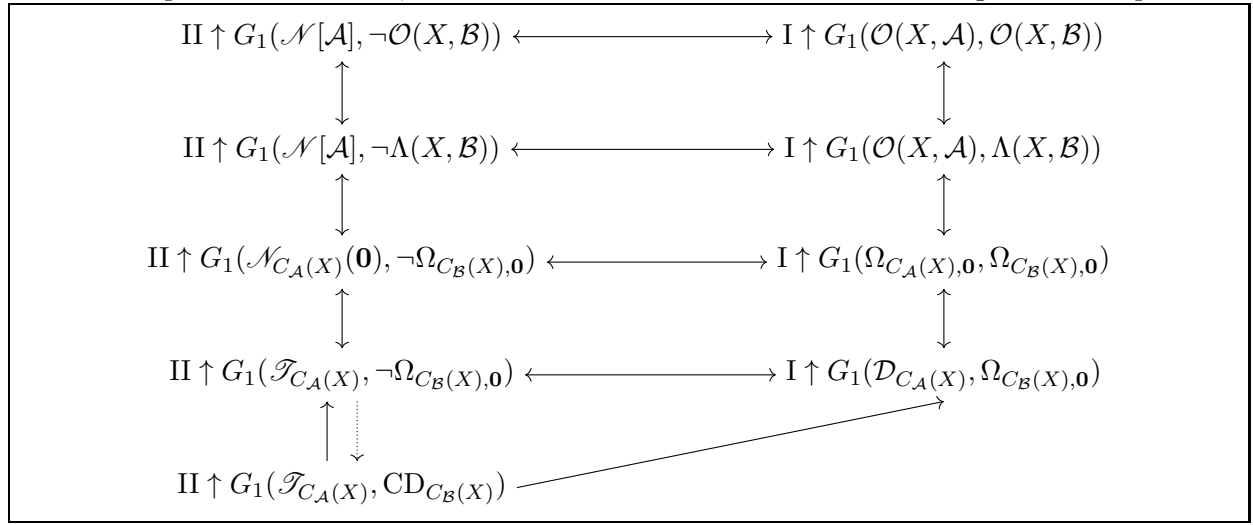
\framebox[\textwidth]{
\begin{tikzpicture}
    \node (pointOpen) at (-4,0) {\(\Two \uparrow G_1(\mathscr N[\mathcal A], \neg \mathcal O(X,\mathcal B))\)};
    \node (pointLambda) at (-4,-1.5) {\(\Two \uparrow G_1(\mathscr N[\mathcal A], \neg \Lambda(X,\mathcal B))\)};
    \node (GruenhageCluster) at (-4,-3) {\(\Two \uparrow G_1(\mathscr N_{C_{\mathcal A}(X)}(\mathbf 0), \neg \Omega_{C_{\mathcal B}(X),\mathbf 0})\)};
    \node (CL) at (-4,-4.5) {\(\Two \uparrow G_1(\mathscr{T}_{C_{\mathcal A}(X)}, \neg \Omega_{C_{\mathcal B}(X),\mathbf 0})\)};
    \node (CD) at (-4,-6) {\(\Two \uparrow G_1(\mathscr{T}_{C_{\mathcal A}(X)},\mbox{CD}_{C_{\mathcal B}(X)})\)};

    \node (Rothberger) at (4,0) {\(\One \uparrow G_1(\mathcal O(X, \mathcal A), \mathcal O(X,\mathcal B))\)};
    \node (RothbergerLambda) at (4,-1.5) {\(\One \uparrow G_1(\mathcal O(X, \mathcal A), \Lambda(X,\mathcal B))\)};
    \node (CFT) at (4,-3) {\(\One \uparrow G_1(\Omega_{C_{\mathcal A}(X),\mathbf 0},\Omega_{C_{\mathcal B}(X),\mathbf 0})\)};
    \node (CDFT) at (4,-4.5) {\(\One \uparrow G_1(\mathcal D_{C_{\mathcal A}(X)},\Omega_{C_{\mathcal B}(X),\mathbf 0})\)};

    \draw [<->] (pointOpen.east) -- (Rothberger.west);

    \draw[<->] (pointOpen.south) -- (pointLambda.north);
    \draw[<->] (Rothberger.south) -- (RothbergerLambda.north);

    \draw[<->] (pointLambda.east) -- (RothbergerLambda.west);

    \draw[<->] (pointLambda.south) -- (GruenhageCluster.north);
    \draw[<->] (RothbergerLambda.south) -- (CFT.north);

    \draw[<->] (GruenhageCluster.south) -- (CL.north);

    \draw[<->] (GruenhageCluster.east) -- (CFT.west);

    \draw[->, densely dotted] ([xshift=.2cm]CL.south) -- ([xshift=.2cm]CD.north);
    \draw[->] ([xshift=-.2cm]CD.north) -- ([xshift=-.2cm]CL.south);
    \draw[<->] (CFT.south) -- (CDFT.north);

    \draw[<->] (CDFT.west) -- (CL.east);

    \draw[->] (CD.east) -- (CDFT.south);
\end{tikzpicture}
}

\noindent
\framebox[\textwidth]
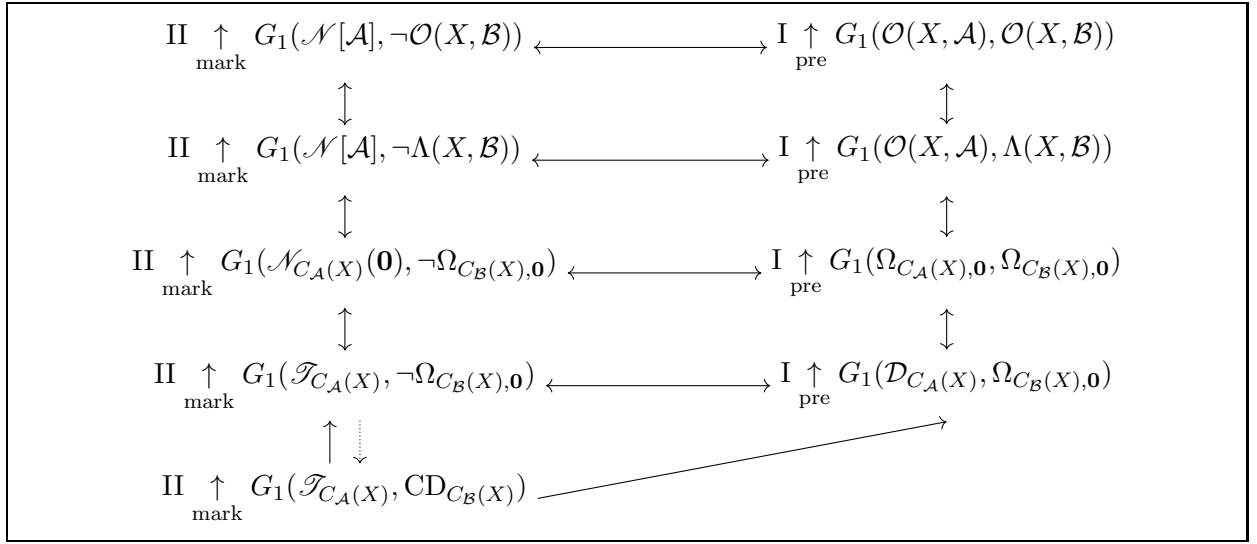
{
\begin{tikzpicture}
    \node (pointOpen) at (-4,0) {\(\Two \underset{\text{mark}}{\uparrow} G_1(\mathscr N[\mathcal A], \neg \mathcal O(X,\mathcal B))\)};
    \node (pointLambda) at (-4,-1.5) {\(\Two \underset{\text{mark}}{\uparrow} G_1(\mathscr N[\mathcal A], \neg \Lambda(X,\mathcal B))\)};
    \node (GruenhageCluster) at (-4,-3) {\(\Two \underset{\text{mark}}{\uparrow} G_1(\mathscr N_{C_{\mathcal A}(X)}(\mathbf 0), \neg \Omega_{C_{\mathcal B}(X),\mathbf 0})\)};
    \node (CL) at (-4,-4.5) {\(\Two \underset{\text{mark}}{\uparrow} G_1(\mathscr{T}_{C_{\mathcal A}(X)}, \neg \Omega_{C_{\mathcal B}(X),\mathbf 0})\)};
    \node (CD) at (-4,-6) {\(\Two \underset{\text{mark}}{\uparrow} G_1(\mathscr{T}_{C_{\mathcal A}(X)},\mbox{CD}_{C_{\mathcal B}(X)})\)};

    \node (Rothberger) at (4,0) {\(\One \underset{\text{pre}}{\uparrow} G_1(\mathcal O(X, \mathcal A), \mathcal O(X,\mathcal B))\)};
    \node (RothbergerLambda) at (4,-1.5) {\(\One \underset{\text{pre}}{\uparrow} G_1(\mathcal O(X, \mathcal A), \Lambda(X,\mathcal B))\)};
    \node (CFT) at (4,-3) {\(\One \underset{\text{pre}}{\uparrow} G_1(\Omega_{C_{\mathcal A}(X),\mathbf 0},\Omega_{C_{\mathcal B}(X),\mathbf 0})\)};
    \node (CDFT) at (4,-4.5) {\(\One \underset{\text{pre}}{\uparrow} G_1(\mathcal D_{C_{\mathcal A}(X)},\Omega_{C_{\mathcal B}(X),\mathbf 0})\)};

    \draw [<->] (pointOpen.east) -- (Rothberger.west);

    \draw[<->] (pointOpen.south) -- (pointLambda.north);
    \draw[<->] (Rothberger.south) -- (RothbergerLambda.north);

    \draw[<->] (pointLambda.east) -- (RothbergerLambda.west);

    \draw[<->] (pointLambda.south) -- (GruenhageCluster.north);
    \draw[<->] (RothbergerLambda.south) -- (CFT.north);

    \draw[<->] (GruenhageCluster.south) -- (CL.north);

    \draw[<->] (GruenhageCluster.east) -- (CFT.west);

    \draw[->, densely dotted] ([xshift=.2cm]CL.south) -- ([xshift=.2cm]CD.north);
    \draw[->] ([xshift=-.2cm]CD.north) -- ([xshift=-.2cm]CL.south);
    \draw[<->] (CFT.south) -- (CDFT.north);

    \draw[<->] (CDFT.west) -- (CL.east);

    \draw[->] (CD.east) -- (CDFT.south);
\end{tikzpicture}
}

\end{theorem}
\begin{proof}
    Since \(\mathcal A\) and \(\mathcal B\) are pre-ideals, the versions of the point-open game are equivalent.

    The arrows between the versions of the point-open game and versions of the Rothberger game come from the duality of the point-open and Rothberger games.
    From this, the versions of the Rothberger game are equivalent.

    Corollaries \ref{corollary:PO=Gru=CD} and \ref{corollary:Gru<PointGamma} generates arrows on the left side of the diagram.
    Similarly, Corollary \ref{corollary:Roth=CFT=CDFT} provides arrows on the right side of the diagram.

    By Corollary, \ref{corollary:CFTDual}, \(\One \underset{\text{pre}}{\uparrow} G_1(\Omega_{C_{\mathcal A}(X),\mathbf 0},\Omega_{C_{\mathcal B}(X),\mathbf 0})\) if and only if \(\Two \underset{\text{mark}}{\uparrow} G_1(\mathscr N_{C_{\mathcal A}(X)}(\mathbf 0), \neg \Omega_{C_{\mathcal B}(X),\mathbf 0})\).

    Proposition \ref{prop:CD->CDFT} adds the implications from the statement \(\Two \underset{\text{mark}}{\uparrow} G_1(\mathscr{T}_{C_{\mathcal A}(X)},\mbox{CD}_{C_{\mathcal B}(X)})\) to \(\One \underset{\text{pre}}{\uparrow} G_1(\mathcal D_{C_{\mathcal A}(X)},\Omega_{C_{\mathcal B}(X),\mathbf 0})\) and then to \(\Two \underset{\text{mark}}{\uparrow} G_1(\mathscr{T}_{C_{\mathcal A}(X)}, \neg \Omega_{C_{\mathcal B}(X),\mathbf 0})\).

    With these connections, the main block of the diagram becomes equivalent without any extra assumptions needed.

    {\color{red}If \(\mathcal A \prec \mathcal B\), Lemma \ref{lemma:Pawlikowski} applies and all of the statements across the two diagrams are equivalent.}
    Correction: If \(\mathcal A = \mathcal B\) are either the finite or compact subsets of \(X\), Lemma \ref{lemma:Pawlikowski}'s revision in \href{https://arxiv.org/abs/2102.00296}{arXiv:2102.00296} applies and all of the statements across the two diagrams are equivalent.
\end{proof}

\begin{note}
Suppose \(\mathcal A = \mathcal B = [\mathbb R]^\omega\).
Define a strategy for One in \(G_1(\mathcal O(X, \mathcal A), \mathcal O(X, \mathcal B))\) as follows:
In the \(n^{\text{th}}\) inning, for any countable set \(A \subseteq \mathbb R\), choose \(U_{A,n}\) to be an open set so that \(A \subseteq U_{A,n}\) and \(U_{A,n}\) has Lebesgue measure \(< 2^{-n}\).
Then \(\sigma(n) = \{ U_{A,n} : A \in \mathcal A\}\).
This is a pre-determined winning strategy for One.

Consider a strategy for Two in \(G_1(\mathcal D_{C_{\mathcal A}(X)}, \Omega_{C_{\mathcal B}(X), \mathbf 0})\) defined as follows:
In the \(n^{\text{th}}\) inning, One's play must have non-trivial intersection with \([\mathbf 0 ; \mathbb Q, 2^{-n}]\).
Let Two choose \(f_n\) in this intersection.
Then as in the previous example, \(f_n \to \mathbf 0\).
This shows that if \(\mathcal A\) does not consist of closed sets, then the properties do not have to be equivalent.
\end{note}
\begin{note}
    If we do not require that \(\mathcal A\) be an ideal base, then the statements
    \begin{itemize}
        \item \(\Two \uparrow G_1(\mathscr N[\mathcal A], \neg \Gamma(X,\mathcal B))\),
        \item \(\One \uparrow G_1(\mathcal O(X, \mathcal A), \Gamma(X,\mathcal B))\), and
        \item \(\Two \uparrow G_1(\mathscr N_{C_{\mathcal A}(X)}(\mathbf 0), \neg \Gamma_{C_{\mathcal B}(X),\mathbf 0})\)
    \end{itemize}
    are all strictly weaker than any of those present in the first diagram of the previous theorem.
    This is also true for Markov/pre-determined strategies.
    The counter example of \(X = \mathbb Z\) with \(\mathcal A\) and \(\mathcal B\) both set to be the singleton subsets of \(\mathbb Z\) demonstrates this.

    Assuming that \(\mathcal A\) is an ideal base makes the situation more complicated.
    In that situation \(\One \uparrow G_1(\mathscr N[\mathcal A], \neg \mathcal O(X,\mathcal B))\) implies that \(\One \uparrow G_1(\mathscr N[\mathcal A], \neg \Gamma(X,\mathcal B))\).
    So to find a space \(X\) where \(\Two \not\uparrow G_1(\mathscr N[\mathcal A], \neg \mathcal O(X,\mathcal B))\) and \(\Two \uparrow G_1(\mathscr N[\mathcal A], \neg \Gamma(X,\mathcal B))\), we need for \(G_1(\mathscr N[\mathcal A], \neg \mathcal O(X,\mathcal B))\) to be undetermined and \(X\) to not be a \(\gamma\)-set.
    These are necessary but not sufficient conditions.
    We do not currently know of any counter-examples, but we also do not know a good reason why the games should be equivalent for player Two.
\end{note}

\section{Applications}

Corollaries \ref{meagerGame} and \ref{nullGame} are direct applications of Lemma \ref{lemma:CoveringPropertyPre}.
\begin{corollary} \label{meagerGame}
    Suppose \(X\) is a space where all closed sets are \(G_\delta\) sets, \(\mathcal A\) consists of the closed nowhere dense sets, and \(\mathcal B\) is the set of all singleton subsets of \(X\).
    Then One has a winning strategy in \(G_1(\mathscr N[\mathcal A], \neg \mathcal O(X,\mathcal B))\) if and only if \(X\) is meager.
\end{corollary}
\begin{corollary} \label{nullGame}
    Suppose \(X\) is a space, \(\mathcal A\) consists of the \(G_\delta\) \(\mu\)-null sets with respect to a Borel measure \(\mu\), and \(\mathcal B\) is the set of all singleton subsets of \(X\).
    Then One has a winning strategy in \(G_1(\mathscr N[\mathcal A], \neg \mathcal O(X,\mathcal B))\) if and only if \(X\) is \(\mu\)-null; i.e., \(\mu\) is the trivial zero measure.
\end{corollary}

The following summarizes a majority of the results from \cite{ClontzHolshouser}.

\begin{theorem} \label{CH1}
    Suppose \(X\) is a Tychonoff space. Then
    \begin{enumerate}[label=(\roman*)]
        \item \label{Group1CH1} \(G_1(\mathscr N[[X]^{<\omega}], \neg\Omega_X)\), \(G_1(\mathscr N_{C_p(X)}(\mathbf 0), \neg \Omega_{C_p(X),\mathbf 0})\), and \(G_1(\mathscr T_{C_p(X)}, \text{CD}_{C_p(X)})\) are equivalent,
        \item \label{Group2CH1} \(G_1(\Omega_X, \Omega_X)\), \(G_1(\Omega_{C_p(X), \mathbf 0}, \Omega_{C_p(X), \mathbf 0})\), and \(G_1(\mathcal D_{C_p(X)}, \Omega_{C_p(X), \mathbf 0})\) are equivalent,
        \item The two groups of games in \ref{Group1CH1} and \ref{Group2CH1} are dual to each other,
        \item \(\text{I} \underset{\text{pre}}{\uparrow} G_1(\mathscr T_{C_p(X)}, \text{CD}_{C_p(X)})\) iff \(X\) is countable iff \(C_p(X)\) is first countable,
        \item For player One, the games \(G_1(\mathscr N[[X]^{<\omega}], \neg\Gamma_X)\) and \(G_1(\mathscr N_{C_p(X)}(\mathbf 0), \neg \Gamma_{C_p(X),\mathbf 0})\) are equivalent to \(G_1(\mathscr N[[X]^{<\omega}], \neg\Omega_X)\) and \(G_1(\mathscr N_{C_p(X)}(\mathbf 0), \neg \Omega_{C_p(X),\mathbf 0})\),
        \item For player Two, \(G_1(\Omega_X, \Omega_X)\) and \(G_1(\Omega_X, \Gamma_X)\) are equivalent,
        \item \(\text{I} \underset{\text{pre}}{\uparrow} G_1(\Omega_X, \Omega_X)\) if and only if \(\text{I} \uparrow G_1(\Omega_X, \Omega_X)\).
    \end{enumerate}
\end{theorem}

The following summarizes a majority of the results from \cite{CaruvanaHolshouser}.

\begin{theorem} \label{CH2}
    Suppose \(X\) is a Tychonoff space. Then
    \begin{enumerate}[label=(\roman*)]
        \item \label{Group1CH2} \(G_1(\mathscr N[K(X)], \neg\mathcal K_X)\), \(G_1(\mathscr N_{C_k(X)}(\mathbf 0), \neg \Omega_{C_k(X),\mathbf 0})\), and \(G_1(\mathscr T_{C_k(X)}, \text{CD}_{C_k(X)})\) are equivalent,
        \item \label{Group2CH2} \(G_1(\mathcal K_X, \mathcal K_X)\), \(G_1(\Omega_{C_k(X), \mathbf 0}, \Omega_{C_k(X), \mathbf 0})\), and \(G_1(\mathcal D_{C_k(X)}, \Omega_{C_k(X), \mathbf 0})\) are equivalent,
        \item The two groups of games in \ref{Group1CH2} and \ref{Group2CH2} are dual to each other,
        \item \(\text{I} \underset{\text{pre}}{\uparrow} G_1(\mathscr T_{C_k(X)}, \text{CD}_{C_k(X)})\) iff \(X\) is hemicompact iff \(C_k(X)\) is first-countable,
        \item For player One, \(G_1(\mathscr N[K(X)], \neg\Gamma_k(X))\) and \(G_1(\mathscr N_{C_k(X)}(\mathbf 0), \neg \Gamma_{C_k(X),\mathbf 0})\) are equivalent to \(G_1(\mathscr N[K(X)], \neg\mathcal K_X)\) and \(G_1(\mathscr N_{C_k(X)}(\mathbf 0), \neg \Omega_{C_k(X),\mathbf 0})\),
        \item For player Two, \(G_1(\mathcal K_X, \mathcal K_X)\) and \(G_1(\mathcal K_X, \Gamma_k(X))\) are equivalent,
        \item \(\text{I} \underset{\text{pre}}{\uparrow} G_1(\mathcal K_X, \mathcal K_X)\) if and only if \(\text{I} \uparrow G_1(\mathcal K_X, \mathcal K_X)\).
    \end{enumerate}
\end{theorem}

Notice that the property of being \(\sigma\)-compact lies in between being countable and being hemicompact.
If we use the fact that Theorems \ref{MD1} and \ref{MD2} apply to pairs \(\mathcal A\) and \(\mathcal B\), then we can generate a setup which characterizes \(\sigma\)-compactness in way that is similar to Theorems \ref{CH1} and \ref{CH2}.

\begin{theorem}
    Suppose \(X\) is a Tychonoff space. Then
    \begin{enumerate}[label=(\roman*)]
        \item \label{Group1} \(G_1(\mathscr N[K(X)], \neg\Omega_X)\), \(G_1(\mathscr N_{C_k(X)}(\mathbf 0), \neg \Omega_{C_p(X),\mathbf 0})\), and \(G_1(\mathscr T_{C_k(X)}, \text{CD}_{C_p(X)})\) are equivalent,
        \item \label{Group2} \(G_1(\mathcal K_X, \Omega_X)\), \(G_1(\Omega_{C_k(X), \mathbf 0}, \Omega_{C_p(X), \mathbf 0})\), and \(G_1(\mathcal D_{C_k(X)}, \Omega_{C_p(X), \mathbf 0})\) are equivalent,
        \item The two groups of games in \ref{Group1} and \ref{Group2} are dual to each other,
        \item \(\text{I} \underset{\text{pre}}{\uparrow} G_1(\mathscr T_{C_k(X)}, \text{CD}_{C_p(X)})\) iff \(X\) is \(\sigma\)-compact iff \(\cof(\mathscr N_{C_k(X)}(\mathbf 0); \mathscr N_{C_p(X)}(\mathbf 0), \supseteq) = \omega\),
        \item For player One, the games \(G_1(\mathscr N[K(X)], \neg\Gamma_X)\) and \(G_1(\mathscr N_{C_k(X)}(\mathbf 0), \neg \Gamma_{C_p(X),\mathbf 0})\) are equivalent to \(G_1(\mathscr N[K(X)], \neg\Omega_X)\) and \(G_1(\mathscr N_{C_k(X)}(\mathbf 0), \neg \Omega_{C_p(X),\mathbf 0})\), and
        \item For player Two, \(G_1(\mathcal K_X, \Omega_X)\) and \(G_1(\mathcal K_X, \Gamma_X)\) are equivalent.
    \end{enumerate}
\end{theorem}

\section{Open Questions}
\begin{itemize}
    \item Is there a topological characterization of the statement \(\cof(\mathcal A; \mathcal B, \leq) \leq_T \omega^\omega\)?
    \item Does \(\text{I} \uparrow G_1(\mathcal K_X, \Omega_X)\) imply \(\text{I} \underset{\text{pre}}{\uparrow} G_1(\mathcal K_X, \Omega_X)\)?
    \item More broadly, to what extent can the Pawlikowski generalization presented here be further generalized?
    \item If \(\mathcal A\) is an ideal base, are \(G_1(\mathscr N[\mathcal A], \neg \Gamma(X,\mathcal B))\) and \(G_1(\mathscr N[\mathcal A], \neg \mathcal O(X,\mathcal B))\) equivalent for player Two?
    \item Can the assumption that \(\mathcal B\) consists of \(\mathbb R\)-bounded sets be removed from Theorems \ref{MD1} and \ref{MD2}?
    \item To what extent can the techniques in this paper be used to study more complex selection principles like the Hurewicz property or the \(\alpha\)-Fr{\'{e}}chet properties?
\end{itemize}

\providecommand{\bysame}{\leavevmode\hbox to3em{\hrulefill}\thinspace}
\providecommand{\MR}{\relax\ifhmode\unskip\space\fi MR }
\providecommand{\MRhref}[2]{%
  \href{http://www.ams.org/mathscinet-getitem?mr=#1}{#2}
}
\providecommand{\href}[2]{#2}


\begin{thebibliography}{10}

\bibitem{CaruvanaHolshouser}
Christopher Caruvana and Jared Holshouser, \emph{Closed discrete selection in
  the compact open topology}, Topology Proceedings \textbf{56} (2020), 25 --
  55.

\bibitem{ClontzDuality}
Steven Clontz, \emph{Dual selection games},
    Topology and its Applications \textbf{272} (2020),
    107056.

\bibitem{ClontzHolshouser}
Steven Clontz and Jared Holshouser, \emph{Limited information strategies and
  discrete selectivity}, Topology and its Applications \textbf{265} (2019),
  106815.

\bibitem{Galvin1978}
Fred {Galvin}, \emph{{Indeterminacy of point-open games.}}, {Bull. Acad. Pol.
  Sci., S{\'e}r. Sci. Math. Astron. Phys.} \textbf{26} (1978), 445--449.

\bibitem{Gartside}
  Paul {Gartside} and Ana {Mamatelashvili}, \emph{The {T}ukey {O}rder and {S}ubsets
  of {\(\omega_1\)}}, {Order} \textbf{35} (2018), no.~1, 139--155.

\bibitem{GerlitsNagy}
J.~Gerlits and Zs. Nagy, \emph{Some properties of {C(X)}, {I}}, Topology and
  its Applications \textbf{14} (1982), no.~2, 151 -- 161.

\bibitem{KocinacSelectedResults}
Lj.D.R. Ko{\v{c}}inac, \emph{Selected results on selection principles},
  Proceedings of the Third Seminar on Geometry and Topology (Tabriz, Iran),
  July 15-17, 2004, pp.~71--104.

\bibitem{MichaelVietoris}
Ernest~A. Michael, \emph{Topologies on spaces of subsets}, Trans. Amer. Math.
  Soc. (1951), no.~71, 152--182.

\bibitem{Pawlikowski1994}
Janusz Pawlikowski, \emph{Undetermined sets of point-open games}, Fundamenta
  Mathematicae \textbf{144} (1994), no.~3, 279--285.

\bibitem{SakaiScheppers}
M.~Sakai and M.~Scheepers, \emph{The combinatorics of open covers}, Recent
  Progress in General Topology III (K.P. Hart, J.~van Mill, and P.~Simon,
  eds.), Atlantis Press, 2014, pp.~751--800.

\bibitem{Telgarsky}
    Ratislav Telg{\'{a}}rsky, \emph{Spaces defined by topological games}, Fundamenta Mathematicae \textbf{88} (1975), no.~3, 193--223.

\bibitem{TkachukFE}
V.~V. Tkachuk, \emph{A {Cp}-{T}heory {P}roblem {B}ook. {F}unctional
  {E}quivalencies}, Springer, Cham, 2016.

\bibitem{Tkachuk2018}
\bysame, \emph{Closed discrete selections for sequences of open sets in
  function spaces}, Acta Mathematica Hungarica \textbf{154} (2018), no.~1,
  56--68.

\bibitem{TkachukGame}
\bysame, \emph{Two point-picking games derived from a property of function
  spaces}, Quaestiones Mathematicae \textbf{41} (2018), no.~6, 729--743.



\end{thebibliography}
\end{document}